\documentclass[11pt,oneside]{amsart}
\usepackage[utf8]{inputenc}
\usepackage[english]{babel}
\usepackage{enumitem}
\usepackage{csquotes}
\usepackage{bbm}
\usepackage{subcaption}
\usepackage{float}
\usepackage{fancyhdr}
\usepackage{etoolbox}
\usepackage{hyperref}
\usepackage[hyperpageref]{}
\usepackage{amsmath, amssymb}
\usepackage{mathrsfs}
\usepackage{palatino} 
\usepackage{comment}

\usepackage{mathtools}

\numberwithin{equation}{section}  


\newcommand{\mycomment}[1]{}

\usepackage{graphicx}
\usepackage{wrapfig}
\usepackage{caption}

\usepackage[top=2.5cm, bottom=2.5cm,left=2.5cm,right=2.5cm]{geometry} 

\usepackage{amsthm}

\theoremstyle{plain}
\newtheorem{theorem}{Theorem}
\newtheorem{corollary}[theorem]{Corollary}
\newtheorem*{corollary*}{Corollary}
\newtheorem{lemma}[theorem]{Lemma}
\newtheorem{proposition}[theorem]{Proposition}

\newtheorem{remark}[theorem]{Remark}
\newtheorem{definition}[theorem]{Definition}

\usepackage[svgnames]{xcolor}

\definecolor{bordeaux}{HTML}{cc0000}

\definecolor{viola}{rgb}{0.6, 0.4, 0.8}
\usepackage{pgfplots}

\newcommand{\be}{\begin{equation}}
	\newcommand{\ee}{\end{equation}}

\newcommand{\er}{Erdős–Rényi~}
\newcommand{\whp}{{w.h.p.~}}
\newcommand{\iid}{{i.i.d.~}}
\newcommand{\rhs}{{r.h.s.~}}
\newcommand{\lhs}{{l.h.s.~}}

\newcommand{\N}{\mathbb{N}}

\newcommand{\R}{\mathbb{R}}
\newcommand{\eps}{\varepsilon}
\DeclareMathOperator{\cD}{\mathcal{D}}
\DeclareMathOperator{\cE}{\mathcal{E}}
\newcommand{\E}{\mathbb{E}}
\newcommand{\p}{\mathbb{P}}
\newcommand{\PP}{\mathbb{P}}
\newcommand{\one}{\textup{\textbf{1}}}
\newcommand{\var}{\mathbb{V}\hspace{-1pt}{\rm ar}}

\newcommand{\conv}{\xrightarrow[n \to +\infty]{}}
\newcommand{\pconv}{\xrightarrow[n \to +\infty]{\quad \p \quad}}
\newcommand{\op}{o_{\p}}

\newcommand{\bin}{\text{Bin}}
\newcommand{\pois}{\text{Pois}}

\newcommand{\tv}[1]{\left\|#1\right\|_{\rm TV}}

\newcommand{\tent}{t_{\textup{ent}}}

\newcommand{\quenchG}{\mathbf{P}^G}
\newcommand{\quench}{\mathbf{P}}
\newcommand{\pan}{\p^{\textup{an}}}
\newcommand{\Equench}{\mathbf{E}}

\newcommand{\semi}{{Q}}

\newcommand{\veps}{V^{\star}_\eps}
\newcommand{\entropy}{\textup{H}}
\newcommand{\mass}{\textup{\textbf{m}}}
\newcommand{\pat}{\mathfrak{p}}
\newcommand{\trex}{\textit{Tx}}
\newcommand{\h}{{h_\eps}}
\newcommand{\s}{{s_\eps}}

\newcommand{\gates}{\mathcal{G}}
\newcommand{\tauj}{\tau_{\textup{jump}}}

\newcommand{\mub}{\mu_{\gates_i}}
\newcommand{\muout}{\mu^{\rm out}_{\gates_i}}

\newcommand{\muinb}{\mu^{\rm in}_{\gates_i}}
\newcommand{\muqs}{\mu^\star_i}

\usepackage{todonotes}

\pgfplotsset{compat=1.18} 

\begin{document}
\title[Mixing trichotomy for RW on directed SBM]{Mixing trichotomy for random walks\\on directed stochastic block models}

\author{Alessandra Bianchi}
\address{Dipartimento di Matematica "Tullio Levi-Civita",
Universit\`a di Padova, Via Trieste 63, 35121 Padova, Italy.}\email{alessandra.bianchi@unipd.it}
\author{Giacomo Passuello}
\address{Dipartimento di Matematica "Tullio Levi-Civita",
	Universit\`a di Padova, Via Trieste 63, 35121 Padova, Italy.}
\email{giacomo.passuello@phd.unipd.it}
\author{Matteo Quattropani}
\address{Dipartimento di Matematica e Fisica,
	Universit\`a degli Studi Roma Tre, Via della Vasca Navale 84, 00146 Roma, Italy.}
\email{matteo.quattropani@uniroma3.it}

\date{\today}
\numberwithin{equation}{section}
	\begin{abstract}
		We consider a directed version of the classical Stochastic Block Model with $m\ge 2$ communities and a parameter $\alpha$ controlling the inter-community connectivity. We show that, depending on the scaling of $\alpha$, the mixing time of the random walk on this graph can exhibit three different behaviors, which we refer to as subcritical, critical and supercritical. In the subcritical regime, the total variation distance to equilibrium decays abruptly, providing the occurrence of the so-called cutoff phenomenon. In the supercritical regime, the mixing is governed by the inter-community jumps, and the random walk exhibits a metastable behavior: at first it collapses to a local equilibrium, then, on a larger timescale, it
can be effectively described as a mean-field process on the $m$ communities, with a decay to equilibrium which is asymptotically smooth and exponential. Finally, for the critical regime, we show a sort of interpolation of the two above-mentioned behaviors. 
Although the metastable behavior shown in the supercritical regime appears natural from a heuristic standpoint,
a substantial part of our analysis can be read as a control on the homogenization of the underlying random environment.


    \par\smallskip\noindent
    {\it Acknowledgments.} All the authors are member of GNAMPA-INdAM, and acknowledge financial support through the GNAMPA projects ``Redistribution models on networks'',  
     ``Ferromagnetism versus synchronization: how does disorder destroy universality?'', and ``Moment problems techniques for particle systems and packing graphs''.
     The work of A.~Bianchi is partially funded by the University of Padova through the BIRD project 198239/19  ``Stochastic processes and applications to disordered systems''.
	\end{abstract}

\maketitle

\section{Introduction}\label{sec:intro}

\subsection{Mixing trichotomy: a phase transition in the mixing behavior}
Phase transitions are among the most interesting and well-studied phenomena in statistical physics. In the broadest sense, a \emph{phase transition} is an abrupt change in behavior of a system as a consequence of a slight variation in its parameters. In the mathematical literature, there are plenty of simple probabilistic models that present such a behavior, so it would be worthless to even try to do a small account here. Let us rather focus on a single, striking example of a phase transition phenomenon arising in the study of (large) finite Markov chains: the \emph{cutoff phenomenon}. In a few words, a Markov chain is said to \emph{exhibit cutoff} if the decay of the distance to equilibrium of the chain---measured, e.g., in total variation---takes place in an abrupt manner: there exists some $t$ (which depends on the size of the system) such that, for any $\varepsilon>0$, at time $(1-\varepsilon)t$ its distribution is arbitrarily far from the equilibrium one, while at time $(1+\varepsilon)t$ it is arbitrarily near. In this sense, the ``parameter'' inducing the phase transition here is \emph{time}, and for this reason it is often said that this is an example of a \emph{dynamical} phase transition. 

Despite an increasing amount of work on the subject, the cutoff phenomenon is still far from being completely understood, and the research of simple conditions (i.e., easy-to-check and model independent) guaranteeing the presence of a cutoff is still very active.
In such a landscape, a quite natural line of investigations lies in considering \emph{parametric perturbations} of Markov chains that exhibit a cutoff. More explicitly, focusing the analysis on a particular choice of chain and a method of implementing the perturbation, it is interesting to quantify the \emph{robustness} of the cutoff phenomenon with respect to the \emph{strength} of the perturbation. 

In this setup, the mixing time of the dynamics may exhibit an additional phase transition, governed by the strength of the perturbation and referred to as a \emph{mixing trichotomy}. 
This involves the identification of a \emph{subcritical} regime, 
where the perturbation is sufficiently weak to preserve the mixing behavior of the original chain, contrasted with a \emph{supercritical} regime, where the ``cutoff picture'' is disrupted, leading to \emph{smooth} rather than \emph{abrupt} mixing.
As in the broader framework of statistical physics, the description of such a phase transition is typically accompanied by the identification of a \emph{critical} regime, at the interphase between the other two, where the system exhibits some sort of intermediate behavior.

Among the models for which such a phenomenon has been rigorously proved, we recall: Glauber dynamics for the Ising model \cite{LSg1,LSg2}, random walks on dynamic graphs \cite{AGHH19,AGHH22,CQ21b}, random walks with reset (and the so-called PageRank dynamics) \cite{CQ21a,VS}, Ehrenfest urns with multi-type particles \cite{Q24}, mass redistribution models \cite{CQS24}. 
In particular, although the models in \cite{AGHH19,CQ21a,CQ21b} are very different from each other, they all share a common feature: in the supercritical regime, the total variation distance over time---properly rescaled---converges to an exponential function. Moreover, in all these examples, such an exponential decay of the distance to equilibrium can be read from the point of view of the trajectory of the process: the arrival at equilibrium is due to the occurrence of a certain event, and the time of the first occurrence of the latter event is (asymptotically) exponentially distributed. From a high-level perspective, this is the cartoon underlying another classical phenomenon in statistical physics, known as \emph{metastability}: the system is trapped in a \emph{local equilibrium} up to some (large) exponential time in which the global equilibrium is eventually reached. In other words, looking through the lens of the \emph{mixing trichotomy approach}, one might be tempted to see metastability and cutoff as ``opposed'' phenomena.\\
To make this idea more convincing, it is worth recalling that the emergence of a cutoff
is often explained in terms of \emph{concentration phenomenon}. 
Conversely, the memorylessness of the exponential distribution, which characterizes metastable behaviors, 
can be seen as the complete opposite of \emph{concentration}. 
The goal of this paper is to make the heuristic picture discussed so far as explicit as possible by means of a simple (but natural) model. 

\subsection{Mixing time of random walks on graphs with a community structure}

In recent years, random walks on random graphs have been extensively
studied on various random graph models, becoming a prototypical example
of Markov chains that exhibit the cutoff phenomenon. 
Notable contributions include the establishment of the cutoff for 
 random walks on regular random graphs \cite{LSrrg}, on the giant component of the Erd\"os-Rényi graph \cite{BLPS}, 
on the configuration model \cite{BLPS,BHS,BHLP}, 
on all Ramanujan graphs \cite{LP16}, and on random lifts \cite{BL22,CK22}.  

A significant part of these investigations has also focused
on the directed setting, which is particularly challenging due to
the non-reversibility of the dynamics and the poor knowledge
of the stationary distribution.
This framework, which is closely related to our work, 
was first explored in \cite{BCS1,BCS2}, 
where the cutoff was established for random walks on the directed configuration model and for a broad class of sparse Markov chains, 
and later extended to the case of heavy-tailed degrees \cite{CCPQ}, to PageRank dynamics \cite{CQ21a}, and to the directed Chung-Lu model \cite{BP}. Other results in the same spirit have also been obtained in \cite{BL22,Dub1,Dub2}.

A common thread in all these works is the characterization of the mixing time of the random walk on the graph
in terms of the entropy production rate (or simply, the \emph{entropy}) of the random walk on its \emph{local weak limit}.
The cutoff time is then shown to correspond
to the entropic time $\tent$ (see Eq.~\eqref{eq:tent} for a formal definition),
 which corresponds to the logarithm of the size of the graph divided by the entropy.

The first extension of this approach to random graphs with many communities is due to Ben-Hamou \cite{Ben-Hamou}, who studied the \emph{non-backtracking} random walk on a variant of the configuration model incorporating a 2-community structure. 
In \cite{HSS}, the authors extended the analysis to the simple random walk and allowed multiple communities.
These results reveal
that the community structure can create a configurational bottleneck
in the set of random walk trajectories, depending on the strength
of interactions between communities, and disrupt the entropic picture. 
In fact, the mixing behavior displays a phase transition
among a subcritical regime, where the interaction strength is sufficiently high so that 
the random walk exhibits cutoff at the same time as in the single-community case,  
and a supercritical regime, where the low interaction strength results in a smooth--rather than abrupt--mixing, and the convergence to equilibrium is driven by the occurrence of the inter-community transitions. In both these works, the intensity of inter-community connections is modeled via a parameter $\alpha\in[0,1]$, and the critical scaling for this parameter is shown to correspond to the inverse of the entropic time of a single community. Nevertheless, in \cite{Ben-Hamou,HSS} the authors do not attempt a refined analysis of the total variation distance profile in the supercritical regime, being their focus on the cutoff/non-cutoff transition.

In this work, our aim is to extend their results to another class of directed random graphs with community structures and also to complete \emph{trichotomy picture} by providing sharp results on the total variation profile in the critical and supercritical cases.

To this end, we will consider a block model with $m > 1$ communities, 
constructed from $m$ independent directed \er random graphs on $n$ vertices (the \emph{communities}) and a random rewiring procedure governed by a parameter
$\alpha \equiv \alpha_n \in [0,1]$, which introduces connections between these communities. The smaller $\alpha$, the sparser the inter-community connectivity.
Specifically, we will work in a \emph{weakly sparse} regime, setting the probability of connection for the \er communities to
$p \equiv p_n := \lambda \log(n) / n$, for a constant $\lambda > 1$, and assuming
that each edge is \emph{rewired} so as to point to a different community with probability $\alpha\gg (\lambda n\log n)^{-1}$. 
On the one hand, 
the choice $\lambda>1$ ensures that, before the rewiring, each community is strongly connected with high probability (see, e.g., \cite{CF}). On the other hand, the requirement $\alpha\gg (n\log n)^{-1}$ is needed for the communities to be typically connected, even though such connections may be sparse.

We will then analyze the convergence to the equilibrium
of a random walk on the resulting graph, letting
the interaction parameter $\alpha$, or rather the sequence $\alpha_n$, $n \in \mathbb{N}$, vary on $[0,1]$.
As a main result, as stated in Theorem \ref{thm:main}, we will show that
the dynamics exhibits a mixing trichotomy with a critical regime at $\alpha^{-1} \asymp \tent$, in line with the findings in \cite{Ben-Hamou,HSS}. 
For higher values of $\alpha$, the system is in a subcritical regime with cutoff at time $\tent$, as described in Eq.~\eqref{eq:sub-cutoff}, while for lower values of $\alpha$, 
the system enters a supercritical regime where the cutoff behavior is disrupted.  
In this regime, we will then identify two relevant timescales.  
\begin{enumerate} 
\item[(i)] 
The first corresponds to the attainment of a local equilibrium within the starting community, which causes a sharp drop in the total variation distance from $1$ to $(m - 1)/m$ at time $\tent$ (see Eq.~\eqref{eq:sup-cutoff}).  
\item[(ii)] The second is a larger timescale of order $\alpha^{-1}$, which governs the convergence to the global equilibrium through a smooth exponential decay of the total variation distance \eqref{eq:sup-meta}.  
\end{enumerate}
At criticality, where the two timescales are of the same order, 
the dynamics exhibits an intermediate behavior, 
as illustrated on the rightmost side of Figure \ref{fig:subcritical-critical} (see Eq. \eqref{eq:crit-cutoffmeta}).
This can be roughly interpreted as a competition between achieving local equilibrium within a community
and the first passage across two communities. We will provide an extended interpretation of this result in Section \ref{suse:results}.

\section{Model and results} \label{sec:model}
In this section, we introduce the necessary notation and preliminaries, 
and we state our main results. 
We begin in Section~\ref{suse:graph-model} by providing a precise definition
of the graph model. Then, in Section~\ref{suse:preliminaries}, we review
key concepts essential for stating and interpreting our main result, Theorem~\ref{thm:main}, which is presented and discussed formally in Section~\ref{suse:results}. 
In Section~\ref{suse:remarks-model}, we comment on our choice of the graph model and explore potential extensions of the framework considered in this paper. The proof of the main result is developed in several sections, from Section~\ref{sec:prel} to Section~\ref{sec:proof}. A detailed road map that describes the structure of this technical part is provided in Section~\ref{suse:organization-proof}.

\subsection{The graph model}\label{suse:graph-model}
We consider the following model, which we call \emph{Directed Block Model} and denote by ${\rm DBM}(n,m,p,\alpha)$:
\begin{enumerate}
\item[(1)] Consider $m\in\N\setminus\{1\}$ independent directed \er random graphs with $n$ vertices and connection probability $p\in(0,1)$, that is, any ordered couple of vertices presents an oriented edge with probability $p$. We call these graphs $G_1=(V_1,E_1),$ $\dots,$ $G_m=(V_m,E_m)$. More precisely, for each $i \le m$, the vertices in $V_i$ will be labeled by the integers in $[n]$, with the superscript $(i)$ identifying their community of membership.
\item[(2)] For each edge of each graph, throw a coin with a success probability $\alpha\in[0,1/2)$. If it is a head, rewire the edge as follows: if the edge is in the graph $G_i$ with $i \le m$ and it goes from $x^{(i)}$ to $y^{(i)}$, for $x,y\in[n]$, remove the edge $(x^{(i)},y^{(i)})$, choose 
$j\in[m]\setminus\{i\}$ uniformly at random, and let the new edge be $(x^{(i)},y^{(j)})$.
\end{enumerate}
We call $G=(V,E)$ such a graph on the $mn$ vertices.
Notice that, rather than a technical constraint, the requirement $\alpha< \frac12$ is a physical assumption that guarantees some sort of \emph{community structure}, since the majority of edges out-going a given vertex point towards its same community (i.e., are not \emph{rewired}).
 Throughout the paper, we will write $\PP$ (resp. $\E$) to denote the probability law (resp. expectation) of the two-step graph generation process just described.
 As usual in the random graph literature, we will be interested in the asymptotic regime in which $n\to\infty$, and all the asymptotic notation will refer to that limit. We will say that an event occurs with high probability (or simply w.h.p.), if the probability of its occurrence is a function of $n$ that converges to $1$ in the limit $n\to\infty$. 
 The dependence of $\alpha$ on $n$ is the main object of the analysis, since we want to understand how the mixing behavior depends on the relation between $\alpha$ and $n$.

In what follows, we will make the following assumptions on the parameters of the model.
\begin{itemize}
\item We will consider $p=\frac{\lambda\log(n)}{n}$ for some $\lambda>1$ to ensure that each graph is strongly connected with high probability (see Section \ref{sec:prel} below).
\item We will also assume $\lambda\asymp 1$, to have a logarithmic average degree. Due to the result in \cite{BP}, this assumption ensures that the random walk in each of the graphs (before the rewiring) exhibits a cutoff with high probability.
\item We will consider $m\asymp 1$.
\end{itemize}

Let us now introduce some further notation. For each vertex $x\in V$, we denote by $D_x^+$ the out-degree of $x$,
and write
$$D_x^+=O_x^++I_x^+\,,$$
where 
\begin{equation}\nonumber
\begin{split}
O_x^+ := \#\{\mbox{out-edges of } x 
\mbox{ pointing to other graphs}\} \\
I_x^+:= \#\{\mbox{out-edges of } x 
\mbox{ not affected by the rewiring}\}
\end{split}
\end{equation}
Similarly, we call $D_x^-$ the in-degree of $x\in V$.

Notice that while the out-degree $D_x^+$ is the same before
and after the rewiring, the in-degree $D_x^-$ could be different,
and for this reason, we introduce the symbol $D_{x,i}^-$, to denote the degree of a vertex $x\in V_i$ before the rewiring procedure. 
It will also be convenient to define the function $c:V\to\{1,\dots, m\}$ that maps each vertex to its community.

We will first consider the random walk on $G_i$ for $i \le m$, i.e., on a graph \emph{before} the rewiring, and for any probability distribution $\mu$ on $V_i$, we write $\quench_\mu^{G_i}(\cdot)$ for its law when the initial position of the walk has distribution $\mu$. Recall that the latter is a \emph{random} law, since it depends on the realization of the environment $G_i$, and for this reason we will refer to it as \emph{quenched law}. If the graph $G_i$ turns out to be strongly connected and aperiodic (which is the case with high probability), we will call $\pi_i$ the stationary distribution of the random walk on $G_i$. 

Similarly, we will consider the random walk on the whole graph $G$ (after the rewiring procedure), and we will denote by $\quench_\mu^{G}(\cdot)$ its law when the initial position of the walk has distribution $\mu$ (on $V$) and $\pi$ is its stationary distribution, if unique.
We will show below that, in the setting we are dealing with, the measures $\pi$ and $(\pi_i)_{i \le m}$ are \whp unique.
To avoid ambiguities, we will conventionally set them to coincide with the uniform distribution (on the corresponding vertex sets) in the unlikely event that they are not uniquely defined.
Let us finally clarify that we consider a discrete-time evolution. For this reason, we will consider times as integers (possibly taking their integer part).

\subsection{Preliminaries}\label{suse:preliminaries}
Recall that, for $i \le m$, $G_i$ denotes an \er random digraph with 
vertex set $V_i$ and connection probability $p= \lambda \frac{\log n}{n}$, with $\lambda>1$ and $\lambda\asymp1$. 
For any $i\le m$, we will let $P_i$ denote the transition kernel of the simple random walk (SRW in short) on the digraph $G_i$, and we will call 
\begin{equation}\label{eq:entropy}
	\entropy\coloneqq \E\left[\log (D_x^+\vee 1)\right] = -\sum_{y\in V_i} \E\left[P_i(x,y) \log (P_i(x,y))\right]
\end{equation} the average row \emph{entropy} of $P_i$. Let us stress that, thanks to the symmetry of the graph model, the above quantity is independent of the choice of $i\le m$ and $x\in V_i$. Moreover, a first order estimate of $\entropy$, for $n\to\infty$, can be easily deduced to be (see \cite{BP}, Prop.~1)
\begin{equation}
	\entropy\sim{\log \log(n)}\,.
\end{equation}
Let us denote the \emph{entropic time} associated to the entropy in \eqref{eq:entropy} with
\begin{equation}\label{eq:tent}
	\tent  \coloneqq \frac{\log(n)}{\entropy}\,.
\end{equation}
The first and second authors proved the following result, showing a \emph{uniform cutoff phenomenon} for the random walk on $G_i$, taking place exactly at the entropic time \eqref{eq:tent}.
\newcommand{\Tup}{T}
\begin{theorem}[\cite{BP}] \label{thm:cutoff}
	Let $\beta>0$ and $\beta\neq 1$. Then for $i \le m$,
	\begin{equation}
		\max_{x \in V_i}
		\big|\|\quench^{G_i}_x(X_{\beta\tent}\in \cdot)-\pi_i\|_{\rm TV}-\one_{\{\beta<1\}}\big|\,
		\pconv0.
	\end{equation}
\end{theorem}
Theorem \ref{thm:cutoff} is in the same spirit as the findings in \cite{BCS1,CCPQ} for the directed configuration model.
In Section \ref{sect:entropic-method}, its proof
will be readapted to analyze the random walk on the whole graph ${\rm DBM}(n,m,p,\alpha)$ when $\alpha$ is large enough, and to show the occurrence of a similar cutoff behavior as described in Eq.~\eqref{eq:sub-cutoff} below.
 We point out that its proof does not require an explicit point-wise knowledge of the stationary measures $(\pi_i)_{i \le m}$. 

\subsection{Main results}\label{suse:results}
 We are now in a good position to present our main results. 

\begin{theorem}\label{thm:main}
    Let $G$ be a realization of the random digraph ${\rm DBM}(n,m,p,\alpha)$ defined in Section \ref{suse:graph-model}, and $\tent$ the entropic time given in \eqref{eq:tent}. 
    The following mixing trichotomy takes place.
	\begin{itemize}
		\item \textbf{Subcritical case (Fig.~\ref{fig:subcritical-critical}):} if $\alpha^{-1}\ll \tent$ and $\alpha \le \frac12$, then, for all $\beta> 0$ with $\beta\neq 1$,
		\begin{equation}\label{eq:sub-cutoff}
			\max_{x \in V}
			\big|\|\quench^{G}_x(X_{\beta \tent}\in \cdot)-\pi\|_{\rm TV}-\one_{\{\beta<1\}}\big|
			\pconv0\,.
		\end{equation}
		\item\textbf{Critical case (Fig.~\ref{fig:subcritical-critical}):} if $\alpha^{-1}\sim C \tent$ for some constant $C>0$, then, for all $\beta>0$ with $\beta\neq 1$,
		\begin{equation}\label{eq:crit-cutoffmeta}
			\max_{x \in V}
			\left|\|\quench^{G}_x(X_{\beta \tent}\in \cdot)-\pi\|_{\rm TV}- \one_{\{\beta < 1
				\}} -
			\tfrac{m-1}{m}{\rm e}^{-\frac{\beta}{C}\frac{m}{m-1}} \one_{\{ \beta> 1 \}} \right|
			\pconv0\,.
		\end{equation}
		\item\textbf{Supercritical case (Fig.~\ref{fig:supercritical}):} if $\alpha^{-1}\gg \tent$ and $\alpha^{-1}\ll \lambda n\log(n)$, then
		\begin{itemize}
			\item (local equilibrium at $\tent$) for any $\beta\neq 1$
			\begin{equation}\label{eq:sup-cutoff}
				\max_{x \in V}
				\left|\|\quench^{G}_x(X_{\beta \tent}\in \cdot)-\pi\|_{\rm TV}- \one_{\{\beta < 1
					\}} -
				\tfrac{m-1}{m} \one_{\{ \beta> 1\}} \right|
				\pconv0\,,
			\end{equation}
			\item (whole mixing at $\alpha^{-1}$) for any $\beta>0$
			\begin{equation}\label{eq:sup-meta}
				\max_{x \in V}
				\left|\|\quench^{G}_x(X_{\beta \alpha^{-1}}\in \cdot)-\pi\|_{\rm TV}- \tfrac{m-1}{m}{\rm e}^{-\frac{\beta m}{m-1}} \right|
				\pconv0\,.
			\end{equation}
		\end{itemize}
	\end{itemize}
\end{theorem}

\begin{figure}[h!]
        \centering
\begin{subfigure}{.475\linewidth}
		\centering
        \begin{tikzpicture}[baseline]
		\begin{axis}[ 
			width=8cm,
			height=6cm,
			axis lines=middle,
			xlabel={$t=\beta\tent$},
			ytick=\empty,
			xtick=\empty,
			xmin=-0.5,
			xmax=5.5,
			ymin=-1,
			ymax=12,
			extra x ticks={1},
			extra y ticks={11,0}, 
			extra x tick labels={$\beta=1$},
			extra y tick labels={\raisebox{0pt}{$1$},\raisebox{17pt}{$0$}},
			samples=100,
			domain=0:17,
			legend style={draw=none,at={(.98,.35)},anchor=south east}
			]
			\addplot [domain=0:0.2, samples=100, color=blue!80!black, line width=1.2pt](x,11);
			\addplot [domain=0.2:0.4, samples=100, color=blue!70!white, line width=1.2pt](x,11);
			\addplot [domain=0.4:0.6, samples=100, color=blue!40!white, line width=1.2pt](x,11);
			\addplot [domain=0.6:0.8, samples=100, color=blue!20!white, line width=1.2pt](x,11);
			\addplot [domain=0.8:0.95, samples=100, color=blue!10!white, line width=1.2pt](x,11);
			\addplot [domain=0.1:10.7, samples=100, color=black!90!white, dotted, line width=1.2pt](1,x);  
			\addplot [only marks, mark= o, color = black, line width=1.2 pt] table {
				1 11
			};
            \addplot [domain=4.09:4.85, samples=100, color=blue!10!white, line width=1.2pt](x,0);
            \addplot [domain=3.33:4.09, samples=100, color=blue!20!white, line width=1.2pt](x,0);
            \addplot [domain=2.57:3.33, samples=100, color=blue!40!white, line width=1.2pt](x,0);
            \addplot [domain=1.81:2.57, samples=100, color=blue!70!white, line width=1.2pt](x,0);
		      \addplot [domain=1.05:1.81, samples=100, color=blue!80!black, line width=1.2pt](x,0);
			\legend{,,,,,,,$m=6$,$m=5$,$m=4$,$m=3$,$m=2$}
		\end{axis}
	\end{tikzpicture}
    	\end{subfigure}%
	\hfill%
	\begin{subfigure}{.475\linewidth}
        \centering
	\begin{tikzpicture}[baseline]
		\begin{axis}[ 
			width=8cm,
			height=6cm,
			axis lines=middle,
			xlabel={$t=\beta\tent$},
			ytick=\empty,
			xtick=\empty,
			xmin=-0.5,
			xmax=5.5,
			ymin=-1,
			ymax=12,
			extra x ticks={1},
			extra y ticks={11,0}, 
			extra x tick labels={$\beta=1$},
			extra y tick labels={\raisebox{0pt}{$1$},\raisebox{17pt}{$0$}},
			samples=100,
			domain=0:17,
			legend style={draw=none,at={(.98,.35)},anchor=south east}
			]
			\addplot [domain=0:0.2, samples=100, color=viola!80!black, line width=1.2pt](x,11);
			\addplot [domain=0.2:0.4, samples=100, color=viola!70!white, line width=1.2pt](x,11);
			\addplot [domain=0.4:0.6, samples=100, color=viola!40!white, line width=1.2pt](x,11);
			\addplot [domain=0.6:0.8, samples=100, color=viola!20!white, line width=1.2pt](x,11);
			\addplot [domain=0.8:0.95, samples=100, color=viola!10!white, line width=1.2pt](x,11);
			\addplot [domain=0.1:10.7, samples=100, color=black!90!white, dotted, line width=1.2pt](1,x);  
			\addplot [only marks, mark= o, color = black, line width=1.2 pt] table {
				1 11
			};
			\addplot [domain=1.05:3.95, samples=500, color= viola!10!white , line width=1.2pt]{
				11*(6 -1)/6*e^(-((x-1)/2)*6/(6 - 1))};
			\addplot [domain=1.05:3.95, samples=500, color= viola!20!white , line width=1.2pt]{
				11*(5 -1)/5*e^(-((x-1)/2)*5/(5 - 1))};
			\addplot [domain=1.05:3.95, samples=500, color= viola!40!white , line width=1.2pt]{
				11*(4 -1)/4*e^(-((x-1)/2)*4/(4 - 1))};
			\addplot [domain=1.05:3.95, samples=500, color= viola!70!white , line width=1.2pt]{
				11*(3 -1)/3*e^(-((x-1)/2)*3/(3 - 1))};
			\addplot [domain=1.05:3.95, samples=500, color= viola!80!black , line width=1.2pt]{
				11*(2 -1)/2*e^(-((x-1)/2)*2/(2 - 1))};
			\legend{,,,,,,,$m=6$,$m=5$,$m=4$,$m=3$,$m=2$}
		\end{axis}
	\end{tikzpicture}
    	\end{subfigure}
	\caption{Plot of the (theoretical) limiting mixing profile in the subcritical case (left) and critical case (right) with $C=2$ and $m=2,3,4,5,6$.}
    \label{fig:subcritical-critical}
\end{figure}
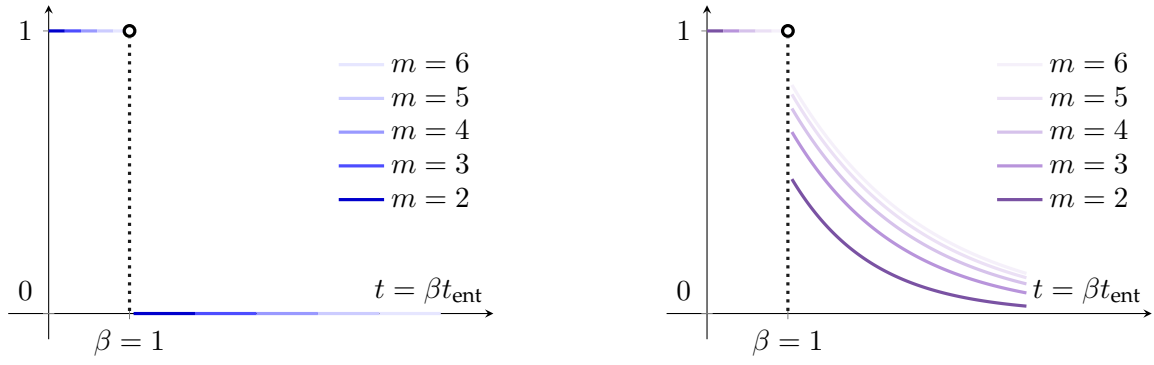

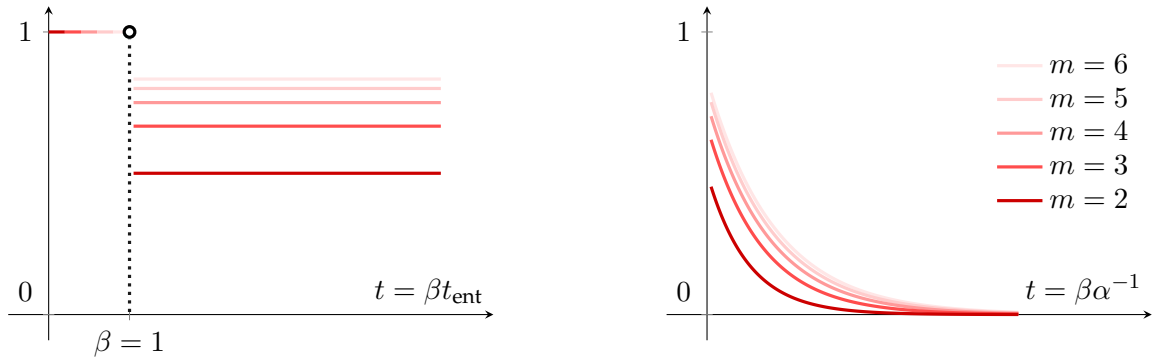
\begin{figure}[!h]
	\centering
	\begin{subfigure}{.475\linewidth}
		\centering
		\begin{tikzpicture}[baseline]
				\begin{axis}[ 
				width=8cm,
				height=6cm,
				axis lines=middle,
				xlabel={$t=\beta\tent$},
				ytick=\empty,
				xtick=\empty,
				xmin=-0.5,
				xmax=5.5,
				ymin=-1,
				ymax=12,
				extra x ticks={1},
				extra y ticks={11,0}, 
				extra x tick labels={$\beta=1$},
				extra y tick labels={\raisebox{0pt}{$1$},\raisebox{17pt}{$0$}},
				samples=100,
				domain=0:17,
				legend style={draw=none,at={(.98,.35)},anchor=south east}
				]
				\addplot [domain=0:0.2, samples=100, color=red!80!black, line width=1.2pt](x,11);
				\addplot [domain=0.2:0.4, samples=100, color=red!70!white, line width=1.2pt](x,11);
				\addplot [domain=0.4:0.6, samples=100, color=red!40!white, line width=1.2pt](x,11);
				\addplot [domain=0.6:0.8, samples=100, color=red!20!white, line width=1.2pt](x,11);
				\addplot [domain=0.8:0.95, samples=100, color=red!10!white, line width=1.2pt](x,11);
				\addplot [domain=0.1:10.7, samples=100, color=black!90!white, dotted, line width=1.2pt](1,x);  
				\addplot [only marks, mark= o, color = black, line width=1.2 pt] table {
					1 11
				};
				\addplot [domain=1.05:4.85, samples=500, color= red!10!white , line width=1.2pt]{11*(6 -1)/6};
				\addplot [domain=1.05:4.85, samples=500, color= red!20!white , line width=1.2pt]{11*(5 -1)/5};
				\addplot [domain=1.05:4.85, samples=500, color= red!40!white , line width=1.2pt]{11*(4 -1)/4};
				\addplot [domain=1.05:4.85, samples=500, color= red!70!white , line width=1.2pt]{11*(3 -1)/3};
				\addplot [domain=1.05:4.85, samples=500, color= red!80!black , line width=1.2pt]{11*(2 -1)/2};
			\end{axis}
		\end{tikzpicture}
	\end{subfigure}%
	\hfill%
	\begin{subfigure}{.475\linewidth}
		\centering
		\begin{tikzpicture}[baseline]
				\begin{axis}[ 
				width=8cm,
				height=6cm,
				axis lines=middle,
				xlabel={$t=\beta\alpha^{-1}$},
				ytick=\empty,
				xtick=\empty,
				xmin=-0.5,
				xmax=5.5,
				ymin=-1,
				ymax=12,
				extra x ticks={},
				extra y ticks={11,0}, 
				extra x tick labels={},
				extra y tick labels={\raisebox{0pt}{$1$},\raisebox{17pt}{$0$}},
				samples=100,
				domain=0:17,
				legend style={draw=none,at={(.98,.35)},anchor=south east}
				]
				\addplot [domain=0.05:3.85, samples=500, color= red!10!white , line width=1.2pt]{
					11*(6-1)/6*e^(-((x))*6/(6 - 1))};
				\addplot [domain=0.05:3.85, samples=500, color= red!20!white , line width=1.2pt]{
					11*(5-1)/5*e^(-((x))*5/(5 - 1))};
				\addplot [domain=0.05:3.85, samples=500, color= red!40!white , line width=1.2pt]{
					11*(4-1)/4*e^(-((x))*4/(4 - 1))};
				\addplot [domain=0.05:3.85, samples=500, color= red!70!white , line width=1.2pt]{
					11*(3-1)/3*e^(-((x))*3/(3 - 1))};
				\addplot [domain=0.05:3.85, samples=500, color= red!80!black , line width=1.2pt]{
					11*(2-1)/2*e^(-((x))*2/(2 - 1))};
				\legend{$m=6$,$m=5$,$m=4$,$m=3$,$m=2$}
			\end{axis}
		\end{tikzpicture}
	\end{subfigure}
	\caption{Plot of the (theoretical) limiting mixing profile in the supercritical case, with $m=2,3,4,5,6$, in the two timescales $t \asymp \tent $ (left) and $t \asymp \alpha^{-1}$ (right).}
    \label{fig:supercritical}
\end{figure}

Theorem \ref{thm:main} represents a neat example of the trichotomy phenomenon described in Section \ref{sec:intro}, and the mechanisms underlying this behavior are easy to read through the mathematical statement and with the help of Figures \ref{fig:subcritical-critical}
and \ref{fig:supercritical}. 
 In the subcritical case, the mixing behavior of the walk is totally unaffected by the presence of a macroscopic community structure, since the inter-community jumps occur on a much shorter timescale compared to the entropic time, which represents the time needed to reach the local equilibrium in a single community. In contrast, in the supercritical phase, the random walk abruptly reaches the local equilibrium of the community where it started, in the time $\tent$. Then, the process essentially behaves as a mean-field random walk on the communities, i.e., as a Markov chain with transition matrix
\begin{equation}\label{eq:trans-complete}
	Q=\left(\begin{matrix} 
		1-\alpha            & \frac{\alpha}{m-1} &     \dots            & \frac{\alpha}{m-1} \\
		\frac{\alpha}{m-1}  &  \ddots            &     \ddots           &  \vdots\\
		\vdots              &  \ddots            &      \ddots          &   \frac{\alpha}{m-1\phantom{|}}\\     
		\frac{\alpha}{m-1}  &  \dots             &  \frac{\alpha}{m-1}  &  1-\alpha 
	\end{matrix}\right)\,,
\end{equation}
which is in fact easily checked to exhibit the mixing behavior in \eqref{eq:sup-meta}.
Finally, in the critical case, the two behaviors interpolate, 
giving rise to the half-cutoff picture, as shown on the right side of Figure~\ref{fig:subcritical-critical}. 
In light of this interpretation, the reader might already foresee
that the proof of Theorem \ref{thm:main} essentially reduces to establishing 
a form of \emph{homogenization} property for the random environment. 
This will be achieved through different methods in the (sub)critical 
and supercritical cases, and by identifying two distinct sub-regimes within the supercritical phase.
We will provide more details on the organization of the proofs in Section \ref{suse:organization-proof}.

\subsection{Some comments on the graph model}\label{suse:remarks-model}
As mentioned in the Introduction, this work aims to establish a simple yet natural framework 
for studying the mixing trichotomy induced by the presence of a bottleneck in the state space. 
In this sense, the model we consider is somehow the \emph{minimal one} to this aim, both in terms of notation and of technicalities required for a rigorous proof. However, it would be possible to extend our findings to more general versions of the model presented in Section \ref{suse:graph-model}. In this section, we present some remarks in the direction of such generalizations.

\subsubsection*{Weakening the assumptions on $\lambda$}
Let us stress that our assumptions on the connectivity parameter $\lambda$ are not expected to be \emph{sharp} to prove a mixing trichotomy. For example, $\lambda>1$ is sufficient but not necessary to guarantee the strong connectivity of communities, and it would be enough to have $(\lambda-1)\log(n)\to +\infty$ as $n\to+\infty$ \cite{CF}.
On the other hand, we do not expect $\lambda\asymp 1$ to be sharp either; rather, we anticipate that any $\lambda=n^{o(1)}$ would yield similar results. In fact, as long as $\lambda=n^{o(1)}$, it should be possible to show 
that the random walk on a single community exhibits a cutoff at the entropic time, which itself forms a divergent sequence in terms of $n$. 
However, since this observation is not rigorously stated in any of the
aforementioned works, we retain this generalization here.

\subsubsection*{Less rigid rewirings}
We chose to define the rewiring procedure in the model as being strictly tied to the ``labels'' of the vertices. The same result can be obtained assuming that the edges are rewired uniformly at random. Although this would require only minor adjustments, it would introduce additional notation to describe the rewiring process. To maintain clarity, we choose to omit this generalization. We believe that the argument presented in the paper is robust to more general non-singular rewiring mechanisms, provided that the rewiring is uniform among the other communities.

\subsubsection*{Removing dependencies between inter- and intra-community out-degrees}
	The construction in Section \ref{suse:graph-model}, introduces, for every $x \in V$, a correlation between the random variables $O_x^+$ and $I_x^+$, which is not present in the context of the classical stochastic block model. 
	However, by randomizing the number of vertices of the graph, it is possible to remove this kind of dependencies. In particular,
assuming that each community has a number $N\sim \pois(n)$ of vertices, for every $n \in \N$, it holds that 
	$O_x^+\sim \bin(D_x^+,\alpha) \sim \pois(n\alpha p)$
	and $I_x^+\sim \bin(D_x^+,1-\alpha) \sim \pois(n(1-\alpha) p)$. Moreover, the last two variables turn out to be independent as desired.
    Being $|N-n|=O(\sqrt{n})$ w.h.p., the arguments given in the proofs remain valid and lead to the same results.

\subsubsection*{Heterogeneous communities}
We believe that, with some extra work, the results presented in this work could be generalized to the case
%
in which the communities have different intra-community connectivities, say $\lambda_1,\dots,\lambda_m$, with $\min_{i \le m}\lambda_i>1$, different rewiring parameters $(\alpha_{i,j})_{i,j \le m}$ with $\max_{j,k\le m}\frac{\alpha_{i,j}}{\alpha_{i,k}}\lesssim1$, or with different sizes, say $n_1,\dots, n_m$, with $\max_{i,j\le m}\frac{n_i}{n_j}\lesssim1$. 
%
%
In particular, while the result in Theorem \ref{thm:main} should remain unaffected by the heterogeneity of the $\lambda_j$'s, we believe that in the latter two cases the matrix $Q$ has to be modified, 
and the (super)critical exponential profile should change to a more general mixing profile, depending on the asymptotic behavior of the corresponding auxiliary Markov chain.
Also notice that when the community sizes are different, the rewiring procedure described in Section \ref{suse:graph-model} is not well defined, and should be suitably modified, e.g., by allowing a rewired edge to point to a uniformly random vertex in the new chosen community, as mentioned above.

\subsubsection*{Diverging number of communities}
Although not the focus of our study, Theorem \ref{thm:main} suggests that even when the number of communities is slowly diverging, the characterization of the critical and supercritical regimes should remain unchanged.
However, if $m \gg1$ grows sufficiently fast, 
the mixing behavior of the model could undergo significant changes. 
For instance, choosing the parameters $(\alpha,m,\lambda)$ so that 
the average number of rewired edges within each community is $\asymp \log(m)$, we expect a cutoff behavior even in the subcritical regime $\alpha^{-1} \gg \tent$. 

In particular, rescaling the time by $\alpha$, we expect that the dynamics 
is well approximated by a simple random walk on a coarse-grained graph obtained by collapsing each community to a single point and erasing multiple edges. 
In that case, the coarse-grained graph is precisely a
directed \er random graph on $m$ vertices in the weakly sparse regime, 
for which cutoff is now well known to take place.
In conclusion, in that setting and for $\alpha^{-1} \gg \tent$, 
we expect to observe a cutoff  on the timescale $t_{{\rm ent}, m}=\alpha^{-1}\log(m)/\entropy_{m}\gg1$, where
$\entropy_{m} \asymp \log \log (m)$ is the row-entropy of the random walk on the coarse-grained graph.

 \subsection{Organization of the proof}\label{suse:organization-proof}
The proof of Theorem \ref{thm:main} is articulated on several parts, depending on the different regimes for the parameter $\alpha$. We now provide a road map to the forthcoming sections.
We start in Section \ref{sec:prel} by presenting some preliminary facts that will be needed throughout the analysis of both the sub- and supercritical case. In Section \ref{sec:super-weak} we deal with the subregion of the
supercritical regime where $\tent\ll\alpha^{-1}\ll\sqrt{n}\log^{-2}(n)$, which we will call \emph{weakly supercritical}. In this region, a control on the total variation distance to equilibrium for both the timescales appearing in Eq.~\eqref{eq:sup-cutoff} and Eq.~\eqref{eq:sup-meta} can be obtained by relying on the analysis of the so-called \emph{annealed random walk}, which will be introduced in Section \ref{sec:prel}. Section \ref{sec:super-strong} is the technical core of our work, and it deals with the remaining subregion of the supercritical case, i.e. $\sqrt{n}\log^{-2}(n)\lesssim\alpha^{-1}\ll \lambda n\log(n)$, which we will refer to as \emph{strongly supercritical}. 
In this case, arguments based on \emph{annealed random walks} 
are doomed to fail, requiring us to employ alternative technical tools
to establish a form of \emph{homogenization property} of the graph. 
This, in turn, offers a level of control over the total variation distance comparable to that achieved in Section \ref{sec:super-weak} for the \emph{weakly supercritical} case.
 Section \ref{sect:entropic-method} deals with the analogue of the results in the previous sections but for the subcritical and the critical case. Here, the arguments can be seen as an adaptation of those in \cite{BCS1,CCPQ,BP} to the DBM case. Finally, in Section \ref{sec:proof} we collect all the estimates obtained throughout the paper and conclude the proof of Theorem \ref{thm:main}.

\section{Approximations and auxiliary processes}\label{sec:prel}
 In this section we aim at presenting some background material and the first statements about the behavior of the random walk on \emph{short timescales}, i.e., for $t\ll\sqrt{n}\log(n)^{-2}$. These results will be used later in the paper in all three regimes.
 In particular, in Section \ref{suse:approx-pi} we present some approximation of the stationary distribution of the \emph{local} random walks in terms of the in-degree sequences; in Section \ref{suse:local-neigh} we recall some classical facts about the local structure of a (weakly) sparse graph seen from a vertex; while in Section \ref{suse:annealed} we introduce the \emph{annealed random walk} and use it as a tool to provide some key approximation for the law of the inter-community jumps performed by the random walk, see Proposition \ref{prop:quenched-law}. 

\subsection{Approximating the stationary distribution}\label{suse:approx-pi}
After the statement of Theorem \ref{thm:cutoff}, we claimed that the proof does not require explicit knowledge of the stationary distribution. Nevertheless, a uniform first-order approximation for the latter in our setup has been obtained by Cooper and Frieze,
as stated by the following theorem.
\begin{theorem}[\cite{CF}] \label{thm:pi-char}
	The local stationary distributions are approximated at first order and uniformly by the (normalized) local in-degree sequences. In formula,
	\begin{equation}
		\max_{i \le m}\max_{x\in V_i}\left| \frac{\pi_i(x)}{\frac{D_{x,i}^-}{pn^2}}-1\right| \pconv 0\,.
	\end{equation}
\end{theorem}
The latter result is partially based on the fact that, given the choice $\lambda>1$, uniform concentration properties for the degrees of vertices of $G$ are true.
\begin{proposition}\label{prop:concentration-degrees}
	There exist two positive constants $C_1$ and $C_2$ such that, \whp,
	\begin{equation}\nonumber
		\max_{i \le m}\max_{x\in V_i}\max\{D_x^+,D_x^-,D_{x,i}^-\} \le C_2\lambda\log(n)\quad\text{and}\quad \min_{i \le m}\min_{x\in V_i}\min\{D_x^+,D_x^-,D_{x,i}^-\} \ge C_1\lambda\log(n)\,.
	\end{equation}
\end{proposition}
\begin{proof}
	The thesis follows immediately from standard concentration bounds for binomial random variables and an application of the union bound.
\end{proof}
As a consequence we may conclude that \whp the stationary distribution $\pi_i$ is uniform, up to corrections bounded away from zero and infinity, and this will be helpful in the characterization of the stationary distribution of the random walk on the whole $G$. 
This fact can then be used to deduce the following byproduct of Theorems \ref{thm:cutoff}, \ref{thm:pi-char}, and Proposition \ref{prop:concentration-degrees}, which allows us to control the $\ell^\infty$-distance to equilibrium of local random walks.
\begin{corollary}\label{coro:mixing}
	Let $S=3 \tent \log(n)$. Then 
	\begin{equation} \label{fast-mixing}
		\max_{i \le m}\max_{x,y\in V_i}
		\left|\frac{\quench^{G_i}_x(X_{S}\in y)}{\pi_i(y)}-1\right|=\op(1) .
	\end{equation}
\end{corollary}
\begin{proof}
	Recall the subadditivity  property of the total variation distance, which implies that
	\begin{equation}
		\PP\left(\max_{i \le m}\max_{x \in V_i}\|\quench^{G_i}_x(X_{2k\tent}\in \cdot)-\pi_i\|\le {\rm e}^{-k}\right)=1-o(1)\,.
	\end{equation}
	Then, by choosing $k=\lfloor\frac32\log(n)\rfloor$,
	\begin{equation} \label{fast-mixing-new}
		\max_{i \le m}\max_{x \in V_i}
		\left\|\quench^{G_i}_x(X_{S}\in \cdot)-\pi_i\right\|_{\rm TV} = \op(n^{-1}) \,,
	\end{equation}
	and the desired result follows by Theorem \ref{thm:pi-char} and Proposition \ref{prop:concentration-degrees}.
\end{proof}
\subsection{Local neighborhoods are tree-like}\label{suse:local-neigh}
The digraphs $(G_i)_{i \le m}$ and $G$ are \whp locally tree-like 
in the following sense: the out-neighborhood $\mathcal{B}^+_{x}({4\eps\tent})$ of depth ${4\eps\tent}$ (growing with $n$) of any vertex $x \in V$, locally look like trees for small $\eps$, as the next proposition clarifies.
	\begin{lemma}\label{lemma:tree-like}
	For a digraph $S=(V,E)$, let $\trex(S):= |E|-|V|+1$ denote the tree-excess of $S$ and, for $\eps>0$,
	\begin{equation} \label{eq:G+}  
			\mathcal{G}_\eps:=\{\forall x \in V,\trex(\mathcal{B}^+_x({4\eps\tent}))< 2 \}.
	\end{equation}
	Let $(X_t)_{t \in \N}$ denote the SRW on $G_i$ or $G$ and, for $t \le {2\eps\tent}$, consider the set and related event
	\begin{align} \label{eq:Veps}
		\veps \coloneqq \left\{ x \in V \mid \trex(\mathcal{B}^+_{x}({4\eps\tent}))=0 \right\}
		, \qquad
		\mathcal{V}_{\eps}\coloneqq \{\max_{x \in V} \quenchG_x(X_t \notin \veps)\le 2^{-t}\}\,.
	\end{align}
	For a sufficiently small $\eps>0$ (independent of $n$), it holds $\p\left(\mathcal{G}_\eps \cap  \mathcal{V}_{\eps}\right) = 1-o(1)$.
\end{lemma}
\begin{proof}
	The proofs for $G_i$ follow \cite[Lemmas 6 and 7]{BP} and can adapted to $G$. 
\end{proof}
This means that out-neighborhoods of depth ${4\eps\tent}$ are at most trees with an extra edge, and if they are not trees, each step of SRW (up to ${2\eps\tent}$ steps) \whp halves the quenched probability to have extra edges.

\subsection{The annealed random walk}\label{suse:annealed}
A key analysis tool is the so-called \emph{annealed random walk}. 
Let us recall its definition and outline its construction in our setting.

By annealed random walk one simply means the non-Markovian process in which the underlying random graph is constructed together with the random walk on it. 
More precisely, at time $t=0$ all vertices are \emph{unrevealed}, in the sense that its out-neighbors are unknown. Then, for any $t\ge0$, assuming $X_t=x\in V$,
\begin{enumerate}
\item If $x$ is revealed, go directly to step (3).
    \item If $x$ is unrevealed, reveal it by generating its out-neighborhood, i.e.:
    \begin{itemize}
    \item[(i)] for each vertex $y\in V_{c(x)}$ toss a coin with success probability $p$;
    \item[(ii)] for each $y\in V_{c(x)}$ which resulted in a success in (i), toss a coin with success probability $\alpha$;
    \item[(iii)] for each $y\in V_{c(x)}$ which resulted in a success in (i) and a failure in (ii), add an arrow from $x$ to $y$;
    \item[(iv)] for each $y\in V_{c(x)}$ which resulted in a success both in (i) and (ii), sample a label $i$ u.a.r.~in $\{1,\dots,m\}\setminus\{c(x)\}$ and add an arrow from $x$ to the vertex in $G_i$  corresponding to $y$ (i.e., the vertex with the same label of $y$).
  \end{itemize}
    \item Let $X_{t+1}$ be one of the out-neighbors of $x$ sampled u.a.r..
 \end{enumerate}
Clearly, such a random process will eventually stop to reveal new vertices. Moreover, the joint law of the revealed vertices coincides with the law $\PP$. The striking feature of the annealed random walk is that it coincides with the expectation of the quenched law, i.e., for any probability distribution $\mu$ on $V$ and $t\ge 0$
\begin{equation}
\pan_{\mu}\big((X_0,\dots,X_t)\in \cdot \big)=\E\big[\quench_\mu^G\big((X_0,\dots,X_t)\in\cdot\big)\big]\,.
\end{equation}
In what follows we will be interested in the event
\begin{equation}
\label{eq:C_t}
\mathcal{C}_t =\{(X_0,\dots,X_t) \text{ is cycle-free}\}\,,
\end{equation}
where by cycle-free we simply mean that the random walk does visit any vertex more than once.

\begin{lemma} \label{lemma:annealed-law}
For any $x\in V$, if $t \ll \sqrt{n}$,
\begin{equation}
	\pan_x(X_t=z,\mathcal{C}_t)=\frac{\semi^t(c(x),c(z))}{n}\left(1+O\left(\tfrac{t^2}{n}\right)\right)\,,\qquad \forall z\in V\,,
\end{equation}
where 
\begin{equation} \label{eq:Q-def}
	\semi^ t(i,j) = \frac{1+(m\one_{i=j}-1)(1-\frac{m}{m-1}\alpha)^{t}}{m}, \qquad i,j \le m, \quad t\ge0,
\end{equation}

\end{lemma}
\begin{remark}
    Notice that \eqref{eq:Q-def} corresponds to the $t$-step transition probability of the Markov chain with transition matrix $Q$ as in \eqref{eq:trans-complete}. This can be verified computing the powers of the diagonal form of $Q$.
\end{remark}
\begin{proof}
Fix $z_0=x\in V$ and $z_t=z\in V$. We will say that a sequence $(z_1,\dots,z_{t-1})$ in $V$ is \emph{cycle free} if $(z_0,\dots,z_{t})$ does not contain repeated vertices.
Start by writing
\begin{equation}
		\begin{split}
			\pan_x(X_t=z , \mathcal{C}_t)&=
			\sum_{\substack{(z_1, \dots, z_{t-1}) \in V^{t-1}\\\text{cycle-free} }} 
			\E\left[\prod_{i=0}^{t-1}\frac{\one_{\{z_i \to z_{i+1}\}}}{D^+_{z_i}}\right]
			\\&
			= 
			\sum_{\substack{(z_1, \dots, z_{t-1}) \in V^{t-1}\\\text{cycle-free} }} 
			\prod_{i=0}^{t-1} \PP(z_i\to z_{i+1})\E\left[\frac{1}{{D}^+_{z_i}}\Big| \one_{\{z_i\to z_{i+1}\}}=1\right].
		\end{split}
	\end{equation}
Conditionally on $\one_{\{z_i\to z_{i+1}\}}=1$, the out-degree of $z_i$ has the same distribution as $1+D$, where $D \sim \bin(n-2,p)$. Then
\begin{equation}
\E\left[\frac{1}{{D}^+_{z_i}}\Big| \one_{\{z_i\to z_{i+1}\}}=1\right]=\frac{1}{np}\left(1+o\left(n^{-1}\right)\right)\,,
\end{equation}
	and as a consequence
	\begin{equation}
		\begin{split}
			\pan_x(X_t=z , \mathcal{C}_t)&
			= 
			\left(1+o\big(\tfrac tn\big)\right)\sum_{\substack{(z_1, \dots, z_{t-1}) \in V^{t-1}\\\text{cycle-free} }} 
			(np)^{-t}\prod_{i=0}^{t-1} \PP(z_i\to z_{i+1})\,.
		\end{split}
	\end{equation}
Moreover, $$
	\PP(z_i\to z_{i+1})=\begin{cases}
		p\frac{\alpha}{m-1}&\text{if }c(z_i)\neq c(z_{i+1})\\
		p(1-\alpha)&\text{if }c(z_i)= c(z_{i+1})
	\end{cases}\,,
	$$ and therefore
		\begin{equation*}
		\begin{split}
			\pan_x(X_t=z , \mathcal{C}_t)&
			= 
			\left(1+o\big(\tfrac tn\big)\right)\sum_{\substack{(z_1, \dots, z_{t-1}) \in V^{t-1}\\\text{cycle-free} }} 
			\frac{1}{n^t}\prod_{i=0}^{t-1} \left(\frac{\alpha}{m-1}\one_{c(z_i)\neq c(z_{i+1})}+(1-\alpha)\one_{c(z_i)= c(z_{i+1})} \right)\,.
		\end{split}
	\end{equation*}
	Since the term within the  brackets is bounded above by $1$ uniformly in $(z_i,z_{i+1})$, we could take the sum over all the possible sequences $(z_1,\dots,z_{t-1})\in V^t$, so having $n^{t-1}$ rather than $(n-2)\cdots(n-t)$ summands. Overall we get
	 \begin{equation*}
	 	\begin{split}
	 		\pan_x(X_t=z , \mathcal{C}_t)&
	 		= 
	 		\left(1+O\big(\tfrac {t^2}n\big)\right)\sum_{\substack{(z_1, \dots, z_{t-1}) \in V^{t-1} }} 
	 		\frac{1}{n^t}\prod_{i=0}^{t-1} \left(\frac{\alpha}{m-1}\one_{c(z_i)\neq c(z_{i+1})}+(1-\alpha)\one_{c(z_i)= c(z_{i+1})} \right)\,,
	 	\end{split}
	 \end{equation*}
and since $t=o(\sqrt{n})$, the prefactor does not change the first order term.
Now we notice that the sum over $V^{t-1}$ collapses into a sum over $[m]^{t-1}$. More precisely, using the shorthand notation $c_i=c(z_i)$, for $i\le t$ and defining 
$$J(c_0,\dots,c_t)=\sum_{j =0}^{t-1}\one_{c_{j+1}\neq c_j}\,,$$ we have
\begin{equation}
\pan_x(X_t=z,\mathcal{C}_t)=\left(1+O\big(\tfrac{t^2}{n}\big)\right)
\frac{1}{n} \sum_{k=0}^{t-1}\sum_{\substack{(c_1,\dots,c_{t-1})\in[m]^{t-1} \\ J(c_0,\dots,c_t)=k}}
\Big(\frac{\alpha}{m-1}\Big)^{k}
\Big(1-\alpha\Big)^{t-k}\,.
\end{equation}
Notice that 
the quantity expressed by the double sum above can be be interpreted as the probability that the Markov chain with transition matrix as in Eq.~\eqref{eq:trans-complete}
is in $c_t=c(z)$ at time $t$.
Therefore
\begin{equation}
\begin{split}
\pan_x(X_t=z,\mathcal{C}_t)
&=\frac{1}{n}\, \semi^ t(c(x),c(z))\left(1+O\left(\frac{t^2}{n}\right)\right)\,. \qedhere
\end{split}
\end{equation}
\end{proof}

Lemma \ref{lemma:annealed-law} can be strengthened to a statement valid for the quenched law of the SRW $(X_t)_{t \ge0}$, provided that we consider sufficiently small times, as it will be shown in the next subsection.
    \begin{proposition} \label{prop:quenched-law}
		For any $\alpha=\alpha_n\in[0,1]$, let $1\ll t \ll \sqrt{n}\log(n)^{-2} $. Then, 
		\begin{equation}\label{eq:315}
		\max_{i \le m}\max_{x \in V}
		\left| \quenchG_x(X_t \in V_i) - \semi^t(c(x),i)\right|
		=\op(1)\,,
		\end{equation}
		where, $\semi^t(c(x),i)$ is defined as in Eq.~\eqref{eq:Q-def}.
	\end{proposition}
		\begin{proof} 
		Notice that, since $m\asymp 1$, it suffices to prove \eqref{eq:315} for a fixed $i\le m$.
            We first consider the case in which $\alpha\ll 1$. Let $K=\lfloor \log^2 n \rfloor$ and call $\cE$ the event that all the vertices in the graph have out-degree at least $C_1\log(n)$ and $\cD$ the event in which each vertex has a tree-excess at most $1$ in its out-neighborhood of height $4$. Recall that, thanks to  Proposition \ref{prop:concentration-degrees} and Lemma \ref{lemma:tree-like}, $\PP(\cD\cap \cE)=1-o(1)$.
			Then
			\begin{equation}
			\p(\one_{\cE\cap\cD}\quenchG_x(X_t\in V_i)\ge \semi^t(c(x),i)+ \delta) \le \frac{\E[\one_{\cE\cap\cD}\quenchG_x(X_t\in V_i)^K]}{(\semi^t(c(x),i)+ \delta)^K}.
			\end{equation}
			Consider now $K$ annealed random walks $(X^{(\ell)})_{\ell\le K}$ starting at some $x\in V$. We let the walks evolve---all starting from $x$---one after the other, for a time $t$. Clearly, these walks are not independent, but each is independent of the previous ones \emph{conditionally on the environment discovered so far}.
			Let us define the family of events $(B_\ell)_{\ell \le K}$. For each $\ell \le K$, $B_\ell$ is the event that:
			\begin{enumerate}
				\item[(i)] $X^{(j)}_t \in V_i$ for each $j \le \ell$;
				\item[(ii)] all the out-degrees of the vertices visited by the first $\ell$ walks are at least $C_1\log(n)$ (cf.~Prop.~\ref{prop:concentration-degrees});
				\item[(iii)] the out-neighborhood of $x$ of height $4$ discovered by the first $\ell$ trajectories has a tree excess at most $1$.
			\end{enumerate}
			With this definitions we have
			\begin{equation}
			\E[\one_{\cE\cap\cD}\quenchG_x(X_t\in V_i)^K] \le \pan_x(B_K)=\pan_x(B_1)\prod_{{\ell}=2}^{K} \pan_x( B_{{\ell}}| B_{{\ell}-1})\,.
			\end{equation}
			We start by noting that, given $B_{{\ell}-1}$, the event $B_{{\ell}}$ is contained in the union of the following events:
			\begin{enumerate}
				\item[(1)] For all times $s\le 4$ the $\ell$-th trajectory is always in a vertex already visited by one of the previous walks.
                Since we are working on the events (ii) and (iii), for any couple $(x,y)$ there are at most $2$ paths of length $4$ joining $x$ to $y$ and each such path has a weight $\le (C_1\log(n))^{-4}$. The probability of the event described above is thus upper bounded by
        $$8K(C_1\log(n))^{-4}=o(1)\,.$$
				\item[(2)] The event in (1) does not occur, i.e., there exists some $s\le 4$ such that the walk visits an unvisited vertex, and there exists a time $s'\in(s,t]$ at which the trajectory intersects again one of the previous walks (including itself); this happens with probability less than 
                $$\frac{Kt^2}{n}=o(1)\,.$$
                \item[(3)] None of the events above occurs, and at time $s=4$ the walk is out of $V_{c(x)}$.  
                Since  $\alpha\ll 1$, this happens with probability at most
                $$1-(1-\alpha)^4=o(1)\,.$$
				\item[(4)] None of the event above occurs, yet at time $t$ the $\ell$-th trajectory is found in $V_i$; thanks to Lemma \ref{lemma:annealed-law}, this happens with probability at most
				\begin{align*}
					\max_{s\le 4}\max_{y\in V_{c(x)}}\pan_y(X_{t-s}\in V_i,\mathcal{C}_{t-s})&=\max_{s\le 4}\semi^{t-s}(c(x),i)\left(1+O\left(\frac{t^2}n\right)\right)\\
				&=\max_{\s\le 4}\semi^{t-s}(c(x),i)+o(1)\\
					&=\semi^{t}(c(x),i)+o(1)\,,
				\end{align*}
                where the first identity comes from the fact that, by symmetry, the annealed probability on the \lhs is independent of $y\in V_{c(x)}$; the second one follows from the fact that $\semi^t(c(x),i)\le 1$ and $t^2\ll n$;
                the third one uses that $t\gg 1$. 
			\end{enumerate}
			In conclusion,  uniformly over ${\ell}\le K$ we have $\pan_x(B_{\ell}|B_{{\ell}-1})=\semi^t(c(x),i)+o(1)$.
			Thanks to the choice of $K$, we conclude that, being $\delta>0$ fixed,
			\begin{equation}
	\p(\one_{\cE\cap\cD}\quenchG_x(X_t\in V_i)\ge\semi^t(c(x),i)+ \delta)
			\le \left(\frac{\semi^t(c(x),i)+o(1)}{\semi^t(c(x),i)+\delta}\right)^K 
			=o(n^{-1}).
			\end{equation}
			Therefore
   			\begin{equation}\label{eq:319}
            \begin{split}
				&\p\big(\max_{x \in V}\quenchG_x(X_t\in V_i)\ge \semi^t(c(x),i)+\delta\big)
				\le \\
                &\qquad\qquad \qquad\p(\one_{\cE\cap\cD}\max_{x \in V}\quenchG_x(X_t\in V_i)\ge \semi^t(c(x),i)+ \delta)+\p(\cE\cup\cD)=o(1)\,.
                \end{split}
			\end{equation}
            			To prove a uniform lower bound on $\quenchG_x(X_t\in V_i)$, one can consider the events
            \begin{equation}
        \bar{\mathcal{E}}_{x,i,\delta}=\left\{ \quenchG_x(X_t\in V_i)\le\semi^t(c(x),i)- \delta\right\}
            \end{equation}
            and
\begin{equation}
\hat{\mathcal{E}}_{x,j,\delta}=\left\{\quenchG_x(X_t\in V_j)\ge\semi^t(c(x),j)+   \frac{\delta}{m-1}\right\}\,.
            \end{equation}
Clearly, for any $i\le m$, $x\in V_i$ and $\delta>0$,  $\cup_{j\neq i}\hat{\mathcal{E}}_{x,j,\delta}\supseteq\bar{\mathcal{E}}_{x,i,\delta} $.
Therefore,
\begin{equation}\label{eq:322}
    \begin{split}
        \p\big(\min_{x \in V}\quenchG_x(X_t\in V_i)\le \semi^t(c(x),i)-\delta\big)&= \p\big(\cup_{x\in V}\bar{\mathcal{E}}_{x,i,\delta}\big)\\
        &\le\p\big(\cup_{j\neq i}\cup_{x\in V}\hat{\mathcal{E}}_{x,j,\delta}\big)\\
        &\le m\max_{j\le m}\p\big(\cup_{x\in V}\hat{\mathcal{E}}_{x,j,\delta}\big)\,.
    \end{split}
\end{equation}
Since $m\asymp 1$ and since the probability on the \rhs of \eqref{eq:322} coincides, replacing $\delta$ with $\frac{\delta}{m-1}$, with the one on the \lhs of \eqref{eq:319}, we conclude that
\begin{equation}\label{eq:323}
    \begin{split}
        \p\big(\min_{x \in V}\quenchG_x(X_t\in V_i)\le \semi^t(c(x),i)-\delta\big)&= o(1)\,.
    \end{split}
\end{equation}
Therefore, in the case $\alpha\ll1$, \eqref{eq:315} follows from \eqref{eq:319} and \eqref{eq:323} and a union bound over $i\le m$.

			 We are left to consider the case $\alpha\asymp 1$. In this case, $Q^t(j,i)=\frac1m+o(1)$ for any $t\gg1$. We can argue as above, but rather than the events (1), (2), (3) and (4), we can consider the events (1), (2) and (4'), where
			 \begin{itemize}
			 	\item[(4')] The events (1) and (2) do not occur, yet at time $t$ the $j$-th trajectory is found in $V_i$; thanks to Lemma \ref{lemma:annealed-law}, this happens with probability at most
			 	\begin{align*}
			 		\max_{s\le 4}\max_{y\in V}\pan_y(X_{t}\in V_i,\mathcal{C}_{t})&=\max_{s\le 4}\max_{j \le m}\semi^{t-s}(j,i)\left(1+O\left(\frac{t^2}n\right)\right)\\
			 		&=\frac1m+o(1)\,.
			 	\end{align*}
			 \end{itemize}
			 This completes the proof.
		\end{proof}

\section{Weakly supercritical regime}\label{sec:super-weak}

In this section and in the following one, we approach the regime $\alpha^{-1} \gg \tent$.

\subsection{First jump across two communities}
We now consider the first time at which the random walk traverses a rewired edge, that is
\begin{equation}\label{eq:def-tau-jump}
	\tauj=\min\{t>0 : c(X_t)\neq c(X_{t-1})\}\,.
\end{equation}
Letting $z_0=x$, and arguing as in the proof of Lemma \ref{lemma:annealed-law} we obtain
\begin{equation} \label{eq:jump-C}
\begin{split}
    \E[\quenchG_x(\tauj > t, \mathcal{C}_t)]
    &=\sum_{\substack{(z_1, \dots, z_{t}) \in V_{c(x)}^{t}\\\text{cycle-free} }} \prod_{i=0}^{t-1} \PP(z_i\to z_{i+1})\E\left[\frac{1}{{D}^+_{z_i}}\Big| \one_{\{z_i\to z_{i+1}\}}=1\right]\\
    &
    =(1-\alpha)^t\left(1+O\big(\tfrac{t^2}n\big)\right).\\    
\end{split}
\end{equation}
We now show that, if $t$ is (twice) the entropic time then, w.h.p.---and uniformly over the starting position---the quenched probability to see a jump to another community before $t$ is small. Before stating the proposition, we need a preliminary lemma that serves as a bootstrap for the forthcoming Proposition \ref{prop:jump-time}.
\begin{lemma}\label{lemma:prelim} If $\alpha\ll 1$, for any constant $a\in\N$,
	\begin{equation}
	\max_{x \in V}\quenchG_x(\tauj \le a) = \op(1)\,.
	\end{equation}
\end{lemma}
\begin{proof}
    Let $\cE$ again denote the event that all the vertices in the graph have out-degree at least $C_1\log(n)$.
	It is enough to prove that
	\begin{equation}\label{eq:ben}
		\max_{x\in V}\frac{O_x^+}{D_x^+}=\op(1)\,.
	\end{equation}
	To see the validity of the estimate in the latter display, we use Bennett's inequality, which gives, for any fixed $\varepsilon>0$
	\begin{equation}
		-\log \PP(O_x^+\ge \varepsilon D_x^+\,,\,\mathcal{E})
        \ge
        -\log \PP(O_x^+\ge \varepsilon C_1 \log(n))
        \gtrsim \varepsilon\log(\varepsilon\alpha^{-1})\log(n)\,.
	\end{equation}
Then \eqref{eq:ben} follows by Proposition \ref{prop:concentration-degrees}, the fact that $\log(\varepsilon\alpha^{-1})\gg 1$, and a union bound.
\end{proof}
\begin{proposition} \label{prop:jump-time}
Let $\alpha^{-1}\gg \tent$. Then for $T \ll \min\{\alpha^{-1}, \sqrt{n}(\log(n))^{-2}\}$,
\begin{equation}\label{eq:prop-jump-time}
    \max_{x \in V}\quenchG_x(\tauj \le  T) = \op(1)\,.
\end{equation}
\end{proposition}

\begin{proof} 
We proceed by the same line of argument as in the proof of Proposition \ref{prop:quenched-law}. Notice that in this case there is no loss of generality in assuming $T\gg 1$, since the (random) map $T\mapsto \quenchG_x(\tauj\le T)$ is deterministically increasing for any choice of $x\in V$. Let $K=\lfloor \log^2 n \rfloor$ and call $\cE$ the event that all the vertices in the graph have out-degree at least $C_1\log(n)$ and $\cD$ the event in which each vertex has a tree-excess at most $1$ in its out-neighborhood of height $4$. Then
	\begin{equation}
	\p(\one_{\cE\cap \cD}\quenchG_x(\tauj\le T)\ge \delta) \le \frac{\E[\one_{\cE\cap \cD}\quenchG_x(\tauj\le T)^K]}{\delta^K}.
	\end{equation}
		Consider now $K$ annealed random walks $(X^{(\ell)})_{\ell\le K}$ starting at some $x\in V$.
		Let us define the family of events $(B_{\ell})_{\ell \le K}$. For each $\ell \le K$, $B_{\ell}$ is the event that:
		\begin{enumerate}
			\item[(i)] $\inf\{s\ge 1\mid c(X^{(j)}_s)\neq c(x) \}\le T$, for all $j\le \ell$;
			\item[(ii)] all the out-degrees of the vertices visited by the first $\ell$ walks are at least $C_1\log(n)$;
			\item[(iii)] the out-neighborhood of $x$ of height $4$ discovered by the first $\ell$ trajectories has a tree excess at most $1$.
	\end{enumerate} 
With this definitions we have
	\begin{equation}
		\E[\one_{\cE\cap \cD}\quenchG_x(\tauj\le T)^K] \le \pan_x(B_K)=\pan_x(B_1)\prod_{{\ell}=2}^{K} \pan_x( B_{{\ell}}| B_{{\ell}-1})\,.
	\end{equation}
	Let us start by noting that
	given $B_{{\ell}-1}$, the event $B_{\ell}$ is contained in the union of the following four events:
	\begin{enumerate}
		\item[(0)] The ${\ell}$-th trajectory jumps to some $y\in V$ such that $c(y)\neq c(x)$ at some time $s\le 4$ (the probability of this event is $o(1)$ thanks to Lemma \ref{lemma:prelim});
		\item[(1)] For all $s\le 4$ the ${\ell}$-th trajectory is always in a vertex already visited by one of the previous walks; thanks to (iii), for any $y$ there are at most $2$ paths of length $4$ joining $x$ to $y$, moreover, thanks to (ii), this happens with probability at most $8K(C_1\log(n))^{-4}=o(1)$;
		\item[(2)] There exists some $s\le 4$ such that the walk visits an unvisited vertex, and there exists a time $s'\in(s,T]$ at which it
             intersects again one of the previous walks (including itself); this happens with probability less than $\frac{KT^2}{n}=o(1)$;
		\item[(3)] None of the event above is verified, yet before time $T$ the ${\ell}$-th trajectory jumps to some $y\in V$ such that $c(y)\neq c(z)$; thanks to Eq.~\eqref{eq:jump-C}, this happens with probability at most $$\pan_x(\tauj\le T ,\mathcal{C}_{T})(1+o(1))=o(1)\,.$$
	\end{enumerate}
	In conclusion, we have $\pan_x(B_{\ell}|B_{{\ell}-1})=o(1)$. Thanks to the choice of $K$, we conclude, for $\delta>0$ fixed, 
	\begin{equation}
	\p(\one_{\cE\cap \cD}\quenchG_x(\tauj\le T)\ge \delta)
	\le \left(\frac{o(1)}{\delta}\right)^K 
	=o\left(n^{-c}\right)\,,\qquad\forall c>0\,.
	\end{equation}
	Therefore
	\begin{equation}
	\p(\max_{x \in V}\quenchG_x(\tauj\le T)\ge \delta)
	\le \p(\one_{\cE\cap \cD}\max_{x \in V}\quenchG_x(\tauj\le T)\ge \delta)+o(1)=o(1)\,.
	\end{equation}
\end{proof}

\subsection{Local equilibrium: a first timescale}
In what follows we will sometimes commit a slight abuse of notation by lifting $\pi_i$ to a probability measure on the entire vertex set $V$.
\begin{theorem} \label{thm:cutoff-community}
	Let $ \alpha^{-1}\gg  \tent$. Then, for any fixed $\varepsilon>0$ and $T \ll \min\{\alpha^{-1}, \sqrt{n}(\log(n))^{-2}\}$,
	\begin{equation}
	\max_{i \le m}\max_{x \in V_i}\max_{t\in[(1+\varepsilon)\tent,T]}\|\quenchG_x(X_{t}\in\cdot)-\pi_i\|_{\rm TV}=\op(1)\,.
	\end{equation}
\end{theorem}

\begin{proof}
	Fix $i \le m$, and notice that there is no loss of generality in assuming $\varepsilon\in(0,1]$. Let $\mathcal{J}=[(1+\varepsilon)\tent,T]$. By the triangle inequality
	\begin{equation} \begin{split}
		\max_{t\in\mathcal{J}}&\max_{x\in V_i}\|\quenchG_x(X_{t} \in \cdot)- \pi_i\|_{\rm TV}\\
		&\le	\max_{t\in\mathcal{J}}\max_{x\in V_i}\|\quenchG_x(X_{t} \in \cdot)- \quench^{G_i}_x(X_{t} \in \cdot)\|_{\rm TV}+\max_{t\in\mathcal{J}}\max_{x\in V_i}\|\quench^{G_i}_x(X_{t} \in \cdot)-\pi_i\|_{\rm TV}\\
		&\le \max_{x\in V_i}\quenchG_x(\tauj\le T)+	\max_{x\in V_i}\|\quench^{G_i}_x(X_{(1+\varepsilon)\tent} \in \cdot)-\pi_i\|_{\rm TV}\,,
	\end{split}
	\end{equation}
	where the second inequality can be deduced by coupling the random walk on $(G_i)_{1\le i \le m}$ and $G$ in the natural way.
	Thanks to Proposition \ref{prop:jump-time} and Theorem \ref{thm:cutoff}, respectively, we may maximize the two terms on the \rhs over $i \le m$ and $x \in V_i$ and get the desired upper bound.
\end{proof}
 In conclusion, Theorem \ref{thm:cutoff-community} shows that for $T \ll \min\{\alpha^{-1},  {n}^{ \frac12 }{\log(n)^{-2}} \}$ the dynamics is trapped in the local equilibrium corresponding to the starting community. 
\subsection{Global equilibrium: a second timescale}\label{suse:global-weak}
In the section, we are going to show that each $T\gg\alpha^{-1}$
provides an upper bound on the mixing time of the SRW on our digraph. 
The claim will be divided into two parts
depending on the value of $\alpha$. In particular, here we will focus on the window $\tent\ll\alpha^{-1}\ll \sqrt{n}\log(n)^{-2}$,
while we postpone to Section \ref{sec:super-strong} the discussion of the strongly supercritical regime, where $\sqrt{n}\log(n)^{-2}\lesssim\alpha^{-1}\ll n\lambda \log(n) $.
As it will be clear along the proofs, the two regimes require different tools and techniques.

	\begin{theorem} \label{thm:pi-approximation}
		Let $\alpha$ be such that $\tent\ll \alpha^{-1}\ll\sqrt{n}{\log(n)^{-2}} $. If $T$ is such that $T \gg \alpha^{-1}$, then
		\begin{equation}
		\max_{x\in V}\left\|\quenchG_x(X_{T}\in\cdot)-\pi\right\|_{\rm TV}=\op(1)\,.
			\end{equation}
		Moreover, it holds $\left\|\frac1m\sum_{i=1}^m{\pi_i}-\pi\right\|_{\rm TV}=\op(1)\,$.
	\end{theorem}
	
	\begin{proof}
Start by fixing $T \ll {n}^{ \frac12 }{\log(n)^{-2}}$. The thesis will hold for general $T$ by monotonicity.
Notice also that there is no loss of generality in replacing $T$ by $T+2\tent$. We use the deterministic bound
		\begin{equation}\label{eq:3}
		\begin{split}
			\left\|\quenchG_x(X_{T+{2\tent}}\in\cdot)-\frac 1m\sum_{i=1}^m  \pi_i\right\|_{\rm TV}&=\left\|\sum_{y \in V}P^T(x,y)P^{2\tent}(y,\cdot)-\frac 1m\sum_{i=1}^m  \pi_i\right\|_{\rm TV} \\
			&=  \left\|\sum_{i=1}^m\left(\sum_{y \in V_i}P^T(x,y)P^{2\tent}(y,\cdot)-\frac{\pi_i}{m}\right)\right\|_{\rm TV}\\
			&\le  \sum_{i=1}^m \left\|\sum_{y \in V_i}P^T(x,y)P^{2\tent}(y,\cdot)-\frac{\pi_i}{m}\right\|_{\rm TV}\,.
		\end{split}
		\end{equation}
		Notice that, thanks to Proposition \ref{prop:quenched-law}, we can bound
		\begin{equation} \label{1/m}
		\max_{i \le m}\max_{x \in V}\left|\quenchG_x(X_T \in V_i)-\frac{1}{m}\right|=\op(1).
		\end{equation}
	Let us fix $i \le m$ and focus on the total variation distance on the \rhs of \eqref{eq:3}. We have
		\begin{equation}
		\begin{split}\label{eq:4}
		&		\max_{i \le m}\max_{x \in V}\left\|\sum_{y \in V_i} P^T(x,y)P^{2\tent}(y,\cdot)-\frac{\pi_i}{m}\right\|_{\rm TV} \\
			&\le 	\max_{i \le m}\max_{x \in V}\left\|\sum_{y \in V_i}P^T(x,y)\left(P^{2\tent}(y,\cdot)-\pi_i\right)\right\|_{\rm TV}+ 	\max_{i \le m}\max_{x \in V}\Big|\quenchG_x(X_T \in V_i)-\frac{1}{m}\Big|\\
			& \le \max_{i \le m}\max_{x \in V}\sum_{y \in V_i}P^T(x,y)\left\|P^{2\tent}(y,\cdot)-\pi_i\right\|_{\rm TV}+\op(1)\\
			& \le	\max_{i \le m}\max_{y\in V_i} \left\|P^{2\tent}(y,\cdot)-\pi_i\right\|_{\rm TV}+ \op(1)=\op(1) \,,
		\end{split}
		\end{equation}
        	where both in the second and in the third inequalities we used \eqref{1/m}, and the last asymptotic bound follows from Theorem \ref{thm:cutoff-community}.  Plugging \eqref{eq:4} into \eqref{eq:3} we deduce that
		 $$\max_{x\in V} \left\|\quenchG_x(X_{T+2\tent}\in\cdot)-\frac1m\sum_{i=1}^m\pi_i\right\|_{\rm TV}=\op(1)\,.$$
		 Then, 
		 \begin{equation} \label{4.18}
		 \begin{split}
		 \left\|\pi-\frac1m\sum_{i=1}^m\pi_i\right\|_{\rm TV} 
		 &\le \sum_{x \in V} \pi(x) \left\|\quenchG_x(X_{T+2\tent}\in\cdot)-\frac1m\sum_{i=1}^m\pi_i\right\|_{\rm TV}\\
		 &\le \max_{x\in V} \left\|\quenchG_x(X_{T+2\tent}\in\cdot)-\frac1m\sum_{i=1}^m\pi_i\right\|_{\rm TV}=\op(1)\,.
		 \end{split}
		\end{equation} 

		This concludes the proof.
	\end{proof}
\section{Strongly supercritical regime}\label{sec:super-strong}

If $ \sqrt{n}\log(n)^{-2}\lesssim\alpha^{-1}\ll \lambda n\log(n)$, the relevant timescale, beyond the scale of $\tent$, is $\alpha^{-1}  \gtrsim  \sqrt{n}\log(n)^{-2}$. We cannot rely on the approximation obtained in Proposition \ref{prop:quenched-law} to control the random walk behavior of such a timescale.
For this \emph{strongly supercritical} regime, we need to generalize the estimates obtained for the \emph{weakly supercritical} regime in Section \ref{sec:super-weak} using a different set of tools. In particular, the following two statements provide the analogue of Proposition \ref{prop:quenched-law} and Theorem \ref{thm:pi-approximation}.

\begin{proposition}\label{prop:quenched-law-new}	For $\sqrt{n}\log(n)^{-2} \lesssim  \alpha^{-1}\ll \lambda n\log(n)$, there exists some $C\gg 1$ such that, if $t\le C\alpha^{-1}$, 
	\begin{equation}\label{eq:part-new-gen}
		\max_{i \le m}\max_{x \in V_i}\left|\quenchG_x(X_t\in V_i)-\semi^t(i,i)\right|=\op(1)\,,
	\end{equation}
    where, $\semi^t(c(x),i)$ is defined as in Eq.~\eqref{eq:Q-def}.
\end{proposition}
\begin{theorem}\label{thm:pi-approximation-new}
For $\sqrt{n}{\log(n)^{-2}}  \lesssim  \alpha^{-1} \ll \lambda n\log(n)$, if $T$ is such that $T\gg \alpha^{-1}$, then
	\begin{equation}\label{eq:pi-approx-1}
		\max_{x\in V}\left\|\quenchG_x(X_{T}\in\cdot)-\pi\right\|_{\rm TV}=\op(1)\,.
	\end{equation}
	Moreover, $\left\|\frac1m\sum_{i=1}^m{\pi_i}-\pi\right\|_{\rm TV}=\op(1)\,.$
    \end{theorem}

As mentioned above, despite the clear analogy with Proposition \ref{prop:quenched-law} and Theorem \ref{thm:pi-approximation}, the proofs of Proposition \ref{prop:quenched-law-new} and Theorem \ref{thm:pi-approximation-new} exploit a different line of argument, based on quasi-stationary distributions, which also shows the emergence of homogenization for such small values of $\alpha$. In order to facilitate the reading, before entering the details, we provide a brief account of the organization of the rest of this section.

\subsubsection*{Organization of the section} The rest of the section is divided into five parts. In Section \ref{suse:gates} we introduce the notion of \emph{gates}. In words, a \emph{gate} is a vertex that has an edge that points toward another community. Clearly, in the strongly supercritical regime, {gates} are rare, since only few vertices have such inter-community connections. With this idea in mind, in Section \ref{suse:quasi-stationary} we provide a control on the first time the random walk visits the set of {gates}. In particular, using the framework of \emph{quasi-stationary distributions}, we show that such a first visit is well approximated by an exponential random variable and characterize its expectation. In Section \ref{suse:coupling} we use the understanding of the hitting time of the set of gates to couple the random walk with a toy process which enjoys some sort of \emph{renewal} property which makes it simpler to analyze. Finally, in Section \ref{suse:proof-prop-thm} we use this coupling to complete the proof of Proposition \ref{prop:quenched-law-new} and Theorem \ref{thm:pi-approximation-new}.

\subsection{The gates}\label{suse:gates}
Fixed a community $i \le m$, the idea is to couple the random walk on $G$ started at some $x\in V_i$ with a simpler process on $V_i$, up to the first time when the random walk moves to another community. To provide further details, it is necessary to first introduce some additional notation.
We we call \emph{gates of $V_i$} the subset of vertices in $V_i$ having at least a rewired out-edge, i.e.,
\begin{equation}
	\gates_i\coloneqq\{y\in V_i\mid  O_y^+>0\}\,.
\end{equation}
In the regime $\alpha^{-1} \gtrsim  \sqrt{n}\log(n)^{-2}$, this set turns out to be small, in the sense that it has a small stationary value, as explained by the next result.
\begin{lemma}\label{lemma:pi-gates}
{If $\alpha^{-1}\gg \lambda\log(n)$},
	\begin{equation}
		\max_{i \le m}\left|\frac{\pi_i(\gates_i)}{\alpha\lambda\log(n)}-1\right|=\op(1)\,.
	\end{equation}
\end{lemma}
\begin{proof}
	For $i \le m$, thanks to Theorem \ref{thm:pi-char}, $\pi_i(\gates_i)$ is \whp well approximated by $\sum_{x \in \gates_i}\frac{D_{x,i}^-}{n^2p}$.
	The latter is a sum of random variables taken on a random set, but it is not difficult to show that it concentrates around its expectation. Indeed, for $k>0$ it holds
	\begin{equation} \label{eq:card-gates}
	\begin{split}
		\p(\big||\gates_i|-\E[|\gates_i |]\big|>k)&\le \frac{\var(|\gates_i|)}{k^2}=  \frac{\sum_{x \in V_i}\var(\one_{\{O_x^+>0\}})}{k^2}\\
		&= (1+o(1))\frac{n( \lambda \alpha \log(n))(1- \lambda \alpha \log(n))}{k^2}\,,
	\end{split}
	\end{equation}
    where we used that, for $x \in V_i$, we have 	
    $ 	\PP(x\in \gates_i)=1-(1-\alpha p)^n
	=\lambda \alpha \log(n)(1+o(1))$.
	Choosing $k$ such that $\sqrt{n \lambda\alpha \log(n)} \ll k \ll n\lambda\alpha \log(n)$, we get that
	\begin{equation} \label{one}
	\p\left(\big||\gates_i|-\lambda n \alpha \log(n)\big|>k\right) = o(1).
	\end{equation}
	Fixed any $\varepsilon>0$,
    using the Chernoff bound, we obtain
\begin{equation}\label{two}
	\begin{split}
		\p&\left(
        \bigg\{\bigg|\sum_{x \in \gates_i}{D_{x,i}^-}-|\gates_i|\lambda \log(n)\bigg|>\eps |\gates_i|\lambda \log(n)\bigg\}  
        \cap \bigg\{\Big||\gates_i|-n \alpha \lambda\log(n)\Big|<k\bigg\}
        \right) \\
		&\le\max_{\delta\in[-\frac1{10},\frac1{10}]} \p\left(\left|\sum_{y=1}^{n\alpha \lambda \log(n) (1+\delta)}{D_{y,i}^-}-n \alpha (\lambda\log(n))^2(1+\delta)\right|>(1+\delta)\eps n \alpha (\lambda\log(n))^2\right)\\
		&\le 2\exp\left\{-\frac{\eps^2 n \alpha (\lambda\log(n))^2}4\right\}\,.
	\end{split}
	\end{equation}
	Notice that the event $\{x \in \gates_i\}$ does not depend on $D_{x,i}^-$, and hence we can rely on the classical bound for the sums of \iid Bernoulli random variables. Then, choosing $\eps=\frac{1}{\log(n)}$ suffices to make the estimate in \eqref{two} vanish. The desired result then follows by combining \eqref{one} and \eqref{two}.
\end{proof}

\subsection{First visit time to gates}\label{suse:quasi-stationary}
We now introduce the so-called \emph{quasi-stationary distribution}, that is, the long-run distribution of the walk on $G_i$ conditioned to the event of not having hit the set $\gates_i$ yet. Let $[P_i]_{\gates_i}$ be the sub-Markovian kernel in which the rows and columns indexed by the vertices in $\gates_i$ have been removed. Then, called $\mathfrak{l}_i$ the largest eigenvalue of $[P_i]_{\gates_i}$, by the Perron-Frobenius theorem there exists a probability distribution $\muqs$ which is a left eigenvector for $[P_i]_{\gates_i}$ associated to the eigenvalue $\mathfrak{l}_i$. In particular, the hitting time of $\gates_i$ for the simple random walk on $V_i$ (with kernel $P_i$) started at $\muqs$ is \emph{exactly} geometrically distributed with parameter $\mathfrak{l}_i$. This $\mathfrak{l}_i$ can also be characterized at first order by the expected hitting time of $\gates_i$ starting at $\pi_i$. All these fact are summarized by the following proposition due to Aldous \cite{A82}.
\begin{theorem}[Cf.~Prop.~A.1 and Lemma A.2 in \cite{SQ}]\label{thm:aldous}
	Let $(W_t)_{t\ge 0}$ be a Markov chain on a finite state space $\Omega$ with transition matrix $\Pi$ and unique stationary distribution $\rho$, fully supported on $\Omega$. Let $\partial\in\Omega$ be a \emph{target state}. Then, there exist a unique probability distribution $\mu_\star$ on 
    $\Omega\setminus \partial$ and a unique $\mathfrak{l} \in (0,1)$ such that
	\begin{equation}
		\lim_{t\to\infty}\quench_\rho \left(W_t=x\mid \tau_{\partial}>t \right)=\mu_\star(x)\,,\qquad \forall x\in\Omega\setminus\partial\,,
	\end{equation}
	and
		\begin{equation}
		\quench_{\mu_\star} \left( \tau_{\partial}>t \right)=(1-\mathfrak{l})^t\,,\qquad \forall t\ge 0\,.
	\end{equation}
	Moreover, observing that $\mathfrak{l}=(\mathbf{E}_{\mu_\star}[\tau_{\partial}])^{-1}$, it holds
	\begin{equation}
		\left|\frac{\mathbf{E}_{\mu_\star}[\tau_{\partial}]}{\mathbf{E}_{\rho}[\tau_{\partial}]}-1 \right|=	\left|\frac{\mathfrak{l}^{-1}}{\mathbf{E}_{\rho}[\tau_{\partial}]}-1 \right|\le \frac{20}{3}\frac{t_{\rm mix}(2+\log\mathbf{E}_{\rho}[\tau_{\partial}])}{\mathbf{E}_{\rho}[\tau_{\partial}]}\,,
	\end{equation}
where
 \begin{equation}
t_{\rm mix}=	\inf\left\{t\ge 0\mid \max_{x\in\Omega}\|\Pi(y,\cdot)-\rho \|_{\rm TV}\le  (2{\rm e})^{-1} \right\}\,.
 \end{equation}
Additionally, consider a sequence of Markov chains with $\Omega_N=[N]$, for $N \in \N$.
If 
\begin{equation}
	\lim_{N\to\infty} T\times \rho(\partial)=0\,,
\end{equation}
where
\begin{equation}\label{eq:T-new}
	T= t_{\rm mix}\times \log\left(1/\min_{x \in\Omega_N}\rho(x)\right)\,,
\end{equation}
then, the expected hitting time of $\partial$, starting at stationarity can be estimated asymptotically by
\begin{equation}\label{eq:est-mean}
	\lim_{N\to\infty}\frac{\mathbf{E}_\rho[\tau_{\partial}]}{R/\rho(\partial)}=1\,,
\end{equation}
where
\begin{equation}\label{eq:R-new}
	R=\sum_{t=0}^{T} {\mathbf{P}}_{\partial}(W_{t} = \partial)\,.
\end{equation}
\end{theorem}  
In our setting, the \emph{target state} is actually a \emph{target set}, namely $\gates_i$. This is not a big issue: indeed, if one considers the Markov kernel $\tilde P_i$ on $\tilde V_i=(V_i\setminus\gates_i)\cup\partial$, where the state $\partial$ represent the merging of $\gates_i$ into a single state and the transitions are set to
\begin{equation}
	\tilde P_i(x,y)=\begin{cases}
		P_i(x,y)&x,y\neq \partial\\
		\sum_{z\in\gates_i}P_i(x,z)&x\neq\partial\text{ and }y=\partial\\
		\sum_{z\in\gates_i}\frac{\pi_i(z)}{\pi_i(\gates_i)}P_i(z,y)&x=\partial\text{ and }y\neq\partial\\
		\sum_{z\in\gates_i}\sum_{v\in\gates_i}\frac{\pi_i(z)}{\pi_i(\gates_i)}P_i(z,v)&x=\partial\text{ and }y=\partial
	\end{cases}\,,
\end{equation}
then one can realize at once that the restriction of $\tilde P_i$ to $V_i\setminus \partial$ coincides with the restriction of $P_i$ to $V_i\setminus\gates_i$. Moreover, by direct computation, it is possible to see that, as soon as $P_i$ has a unique stationary distribution $\pi_i$, then $\tilde P_i$ has a unique stationary distribution $\tilde \pi_i$ satisfying
\begin{equation}
	\tilde \pi_i(x)=\begin{cases}
		\pi_i(x)&x\neq\partial\\
		\pi_i(\gates_i)&x=\partial
	\end{cases}\,.
\end{equation}
 For each $i=1,\dots,m$, we take the choice $\Pi=\tilde P_i$.  
Replacing $\rho(\partial)$ by $\pi_i(\gates_i)$, and $R$ by
\begin{equation}\label{eq:def-R-i}
	\tilde R_i=1+\sum_{t=1}^{\tilde T_i}\tilde P_i^t(\partial,\partial)\,,
\end{equation}
where
\begin{equation}
	\tilde T_i=	\tilde{t}_{{\rm mix},i} \times\log\bigg(1/\min_{x\in \tilde V_i}\tilde{\pi}_i(x)\bigg)\,,
\end{equation}
and
\begin{equation}\label{eq:def-tilde-tmix}
	\tilde{t}_{{\rm mix},i}= \inf\left\{t\ge 0\,\bigg\rvert \max_{x\in \tilde V_i}\|\tilde{P}_i^t(x,\cdot)-\tilde{\pi}_i \|_{\rm TV}\le  (2{\rm e})^{-1} \right\},
\end{equation}
we may estimate the quantity
\begin{equation}
	\mathbf{E}_{\pi_i}[\tau_{\gates_i}]=\tilde{\mathbf{E}}_{\tilde\pi_i}[\tau_{\partial}]\,
\end{equation}
as in \eqref{eq:est-mean}, and provide the following rewriting of Theorem \ref{thm:aldous}, in which we achieve a first-order asymptotic approximation of the rate $\mathfrak{l}$, and which will be proved in the following subsection. 
\begin{proposition}\label{prop:rate}
	With high probability, for every $i \le m$ there exists a probability distribution $\muqs$, supported on $V_i \setminus \gates_i$, and a real number $\mathfrak{l}_i$
	such that
	\begin{equation}
		\quench^{G_i}_{\muqs} \left( \tau_{\gates_i}>t \right)=(1-\mathfrak{l}_i)^t\,,
	\end{equation}
	where $\mathfrak{l}_i$ satisfies
	\begin{equation}\label{eq:rate}
		\mathfrak{l}_i=(1+\op(1)) \lambda\alpha\log(n)\,.
	\end{equation}
	In particular
	\begin{equation}
		\max_{i \le m}\sup_{t \ge 0}\left|\quench^{G_i}_{\muqs} \left(\lambda\alpha \log(n)\cdot\tau_{\gates_i}>t \right)-{\rm e}^{-t}\right|=\op(1)\,.
	\end{equation}
\end{proposition}
To prove Proposition \ref{prop:rate}, we 
 provide a control on the objects that appeared above.
\begin{lemma}\label{lemma:T}
	Recalling the definition in \eqref{eq:def-tilde-tmix}
	\begin{equation}\label{eq:est-T}
		\PP\big(\max_{i \le m}\tilde t_{{\rm mix},i}\le 6\tent\log(n)\big)=1-o(1)\,.
	\end{equation}
\end{lemma}
\begin{lemma}\label{lemma:R-new}
	Recalling the definition in \eqref{eq:def-R-i}
	\begin{equation}
		\max_{i \le m}\tilde R_i=1+\op(1).
	\end{equation}
\end{lemma}
\begin{proof}[Proof of Proposition \ref{prop:rate}]
Plugging the estimates in Lemmas \ref{lemma:pi-gates}, \ref{lemma:T}, and \ref{lemma:R-new} into Theorem \ref{thm:aldous} (and recalling Theorem \ref{thm:pi-char} and Proposition \ref{prop:concentration-degrees}), Proposition \ref{prop:rate} follows at once.
\end{proof}
To conclude, we just need to prove the two lemmas. To this aim and to ease the notation, in what follows, we will call 
\begin{equation} \label{eq:5-26}
	\mub(x)=\frac{\pi_i(x)\one_{x\in\gates_i}}{\pi_i(\gates_i)}\,,\qquad x\in V_i\,,
\end{equation}
the restriction of $\pi_i$ to $\gates_i$ and we set
\begin{equation}\label{eq:muout}
	\muout(x)=\mub P_i(x)\,,\qquad x\in V_i\,.
\end{equation}
The latter is the distribution after one step starting at a random gate sampled with probability proportional to $\pi_i$.
With this notation

 \begin{equation}
	\tilde{R}_i=\sum_{t=0}^{\tilde{T}_i}{\quench}^{G_i}_{\mub}(\bar X_t\in\gates_i)=1+\sum_{t=1}^{T_i}{\quench}^{G_i}_{\mub}(\bar X_t\in\gates_i)\,,
\end{equation}
where $(\bar X_t)_{t\ge 0}$ is the process in which the transition probabilities out of any vertex in $\gates_i$ are set to $\muout$, and the other transition probabilities are the same as $\tilde X$ (and $X$).
Notice that $\tilde X$ is a projection of $\bar X$, and clearly, 
\begin{equation}
	{\quench}^{G_i}_{\mub}(\bar X_t\in\gates_i)={\quench}^{G_i}_{\partial}(\tilde X_t=\partial)\,,\qquad \forall t\ge 0\,.
\end{equation}
The proofs of Lemmas \ref{lemma:T} and \ref{lemma:R-new} rely on the following technical estimate, which will be immediately proved.
\begin{lemma}\label{lemma:key}
	Let
	\begin{equation}\label{eq:bar-tau}
		\bar\tau_{\gates_i}=\inf\left\{t\ge 0\mid \bar X_t\in\gates_i \right\}\qquad\text{and}\qquad 	\bar\tau_{\gates_i}^+=\inf\left\{t\ge 1\mid \bar X_t\in\gates_i \right\}\,.
	\end{equation}
	Then
\begin{equation}\label{eq:return-to-gates}
	\max_{i \le m}{\quench}^{G_i}_{\mub} (\bar\tau_{\gates_i}^+\le \log(n)^3)=\op\big(n^{-1/3}\big)\,.
\end{equation}
\end{lemma}

\begin{proof}[Proof of Lemma \ref{lemma:key}]
		We consider the following partial construction of $G_i$:
		\begin{enumerate}
			\item For each $x\in V_i$, sample a random variable $D_x^+$ with distribution ${\rm Bin}(n-1,p)$.
         	\item Attach, to each $x\in V_i$, $D_x^+$ tails of arrows.
			\item To each tail, attach a ${\rm Bern}(\alpha)$ random variable and call \emph{marked} a tail for which such a random variable is $1$.
         
		\end{enumerate}
		We call $\Sigma$ the set of all possible realizations generated by the randomness just described.  
Notice that, in order to end up with a (sub-)graph having the correct distribution, it is enough to select, for each vertex $x\in V_i$ with $D_x^+$ tails the same number of (distinct) vertices in $V_i\setminus\{x\}$. 
Notice also that the set of gates is fully determined by $\Sigma$, since it coincides with the set of vertices having at least one marked tail.

		Let $\mathcal{M}$ be the set of probability measures on $\gates_i$ defined as
        \begin{equation} \label{eq:M}
			\mathcal{M}=\left\{\mu\in\mathcal{P}(\gates_i)\mid \max_{x,y\in\gates_i}\frac{\mu(x)}{\mu(y)}\le C \right\}\,,
		\end{equation}
		for some bounded $C>0$ that depends on $C_1, C_2$ in Proposition \ref{prop:concentration-degrees}. It holds
		\begin{equation}
			\label{eq:goal-new-2}
			\p(\mub\not\in\mathcal{M})=o(1)\,.
		\end{equation}
		where $\mub$ is the probability measure defined in Eq.~\eqref{eq:5-26}.
		This is an immediate consequence of Theorem \ref{thm:pi-char} and Lemma \ref{lemma:pi-gates}.
		Let $\mathcal{F}\subseteq\Sigma$ be the set of realizations such that the event in Proposition \ref{prop:concentration-degrees} concerning $(D_x^+)_{x\in V_i}$ is satisfied, and such that
		$\frac12 n\alpha\lambda\log(n)\le |\gates_i|\le 2n\alpha\lambda\log(n)\,.$
		Thanks to Proposition \ref{prop:concentration-degrees} and Eq. \eqref{eq:card-gates} in the proof of Lemma \ref{lemma:pi-gates}, it holds 
		\begin{equation} \label{eq:F-complement}
			\p(\mathcal{F}^\complement)=o(1).
		\end{equation}
		We will later show that
		\begin{equation}
			\label{eq:goal-new-1}
			\max_{\sigma\in\mathcal{F}}	\p\left(\max_{\mu\in\mathcal{M}}{\quench}^{G_i}_{\mu} (\bar\tau_{\gates_i}^+\le \log(n)^3) > n^{-1/3}\mid \sigma\right)=o(1)\,,
		\end{equation}
		but let us first point out how the desired result can be derived from  \eqref{eq:goal-new-1}:
		\begin{equation*}\label{eq:518}
			\begin{split}
				\p&\left({\quench}^{G_i}_{{\mub}} (\bar\tau_{\gates_i}^+\le \log(n)^3)>n^{-1/3}\right)\\
				&\le	\p\left({\quench}^{G_i}_{{\mub}} (\bar\tau_{\gates_i}^+\le \log(n)^3) >n^{-1/3} , \mathcal{F}\right)+\p(\mathcal{F}^\complement)\\
				&\le	\p\left({\quench}^{G_i}_{{\mub}} (\bar\tau_{\gates_i}^+\le \log(n)^3) >n^{-1/3} , \mathcal{F}, \mub\in\mathcal{M}\right)+\p(\mub\not\in\mathcal{M})+\p(\mathcal{F}^\complement)\\
				&\le \max_{\sigma\in\mathcal{F}}	\p\left(\max_{\mu\in\mathcal{M}}{\quench}^{G_i}_{{\mu}} (\bar\tau_{\gates_i}^+\le \log(n)^3) > n^{-1/3}\mid \sigma\right)+\p(\mub\not\in\mathcal{M})+\p(\mathcal{F}^\complement)
				=o(1)\,.
			\end{split}
		\end{equation*}
		The three terms on the \rhs of the latter display vanish due to \eqref{eq:goal-new-1}, \eqref{eq:goal-new-2}, and \eqref{eq:F-complement}, respectively. 
      This would complete the proof of \eqref{eq:return-to-gates}.
		
		We are left to prove \eqref{eq:goal-new-1}. Observe that if $\mu \in \mathcal{M}$, then $\max_{x \in \gates_i}\mu(x)\le \min_{y \in \gates_i}\mu(y)C\le C/|\gates_i|$. As a consequence, for any $\sigma\in\mathcal{F}$ and all $n$ large enough, so that $C < \log(n)$,
		\begin{equation}\label{eq:522}
			\p\left(\max_{\mu\in\mathcal{M}}{\quench}^{G_i}_{{\mu}} (\bar\tau_{\gates_i}^+\le \log(n)^3) > n^{-1/3}\mid \sigma\right)\le \p\left(\sum_{x\in\gates_i}\frac{{\quench}^{G_i}_{{x}} (\bar\tau_{\gates_i}^+\le \log(n)^3)}{|\gates_i|} > \frac1{n^{1/3}\log(n)}\mid \sigma\right)\,.
		\end{equation}
		Write $\mu_u$ for the uniform distribution on $\gates_i$, and consider the chain $(\bar{\bar{X}})_{t\ge0}$ in which, when visiting $\gates_i$, the chain is instantaneously set at some vertex in $\gates_i$ u.a.r.. Clearly,
		\begin{equation}\label{eq:541}
			{\quench}^{G_i}_{\mu_u}(\bar\tau_{\gates_i}^+\le \log(n)^3)=\quench^{G_i}_{\mu_u}(\bar{\bar\tau}_{\gates_i}^+\le \log(n)^3)\,,
		\end{equation}
		where $\bar{\bar{\tau}}_{\gates_i}^+$ is the analogue of $\bar{\tau}_{\gates_i}^+$ for the chain $(\bar{\bar X}_{t})_{t\ge 0}$.
		By Markov's inequality, 
		\begin{equation}\label{eq:517b}
			\max_{\sigma\in\mathcal{F}}	\p\left(\quench^{G_i}_{{\mu_u}} (\bar{\bar{\tau}}_{\gates_i}^+\le \log(n)^3) > \frac1{n^{1/3}\log(n)}\mid \sigma\right) \le\log(n)n^{1/3}\max_{\sigma\in\mathcal{F}} \E[\quench^{G_i}_{{\mu_u}} (\bar{\bar{\tau}}_{\gates_i}^+\le \log(n)^3)\mid \sigma]\,.
		\end{equation}
		Call $\pan_{\sigma}$ the law of the annealed walk conditioned to $\sigma$ and with starting distribution $\mu_u$, which is defined as follows:
		\begin{enumerate}
			\item[(i)] Select an element of $\gates_i$ u.a.r.;
   			\item[(ii)] select one of its tails u.a.r., match it to a vertex u.a.r.~in $V_i$ and move the random walk to such a vertex;
			\item[(iii)] as soon as the vertex currently visited by the walk is not in $\gates_i$ continue in this fashion, but if the selected random tail  is already matched, then simply follow the edge while, if it is not matched, match it to a vertex chosen u.a.r.~among those that are not currently connected to the present vertex yet, and move the random walk accordingly;
			\item[(iv)] if instead the vertex visited by the walk at time $r\ge 1$ is some $v\in\gates_i$, then select a vertex $w\in\gates_i$ with probability $\mu_u$, one of its tail u.a.r., and move the random walk as described above. 
            In this case, set $\bar{\bar{\tau}}_{\gates}^{+,\rm an}=r$.
		\end{enumerate}
		Notice that, for any $\sigma$,
		\begin{equation}\label{eq:543}
			\E[\quench^{G_i}_{{\mu_u}} (\bar{\bar{\tau}}_{\gates_i}^+\le \log(n)^3)\mid \sigma]=\pan_{\sigma}(\bar{\bar{\tau}}_{\gates}^{+,\rm an}\le \log(n)^3 )\,.
		\end{equation}
		              		Let $\mathcal{A}_t$ be the event in which there exists two distinct times $s,s'\in[1,t]$ such that the random walk visits the same vertex at $s$ and $s'$.
        Since we are focusing on $\sigma\in\mathcal{F}$, for $t= \log(n)^3$, and large $n$, it holds
		\begin{equation}\label{eq:525}
			\begin{split}
				\pan_{\sigma}(\bar{\bar{\tau}}_{{\gates_i}}^{+,\rm an}\le t)&\le
                \pan_{\sigma}(\bar{\bar{\tau}}_{{\gates_i}}^{+,\rm an}\le t, \mathcal{A}_t^\complement)+\pan_{\sigma}(\mathcal{A}_t)\\
				&\le t\frac{|\gates_i|}{n}+ \frac{t^2}n
				\le \frac{2n\alpha \lambda\log(n)^4+\log(n)^6}{n} =o(n^{-1/2}\log^7(n))\,,
			\end{split}
		\end{equation}
 		and therefore \eqref{eq:goal-new-1} follows at once by \eqref{eq:522}, \eqref{eq:541}, \eqref{eq:517b}, \eqref{eq:543} and \eqref{eq:525}.       
\end{proof}
We can now prove Lemmas \ref{lemma:T} and \ref{lemma:R-new}.
\begin{proof}[Proof of Lemma \ref{lemma:T}] Consider again the process $(\bar X_t)_{t\ge 0}$ in which the transition probabilities out of any vertex in $\gates_i$ are given by the distribution $\muout$ defined in Eq.~\eqref{eq:muout}, and denote with $\bar P_i$ its transition kernel.
Called $\bar\pi_i$ the stationary distribution associated to $\bar P_i$, and defined
	\begin{equation}\label{eq:def-bar-tmix}
		\bar{t}_{{\rm mix},i}= \inf\left\{t\ge 0\,\bigg\rvert \max_{x\in  V_i}\|\bar{P}_i^t(x,\cdot)-\bar{\pi}_i \|_{\rm TV}\le  (2{\rm e})^{-1} \right\}\,,
	\end{equation}
	it is immediate that $\tilde{t}_{{\rm mix},i}\le \bar{t}_{{\rm mix},i}$, since the process $\tilde X$ is a projection of $\bar X$. Notice also that, thanks to Lemma \ref{lemma:pi-gates},
	\begin{equation}\label{eq:pi-pibar}
		\left\|\bar{\pi}_i-\pi_i\right\|_{\rm TV}\le \pi_i(\gates_i)=\op(1)\,.
	\end{equation}
	
We are interested in bounding, for ${S}=3\tent\log(n)$,
\begin{equation}\label{eq:mix1}
	\max_{x \in V_i}\left\|\bar{P}_i^{2S}(x,\cdot)-\bar\pi_i\right\|_{\rm TV}\,.
\end{equation}

Let now, for $A \subset V_i $ and $x \in V_i$,
\begin{equation}
\begin{split}
    \mathcal{O}_{x}(A)
    &\coloneqq\left|{\quench}^{G_i}_x(\bar X_{2S}\in A,\bar{\tau}_{\gates_i}\le 2S)-\quench^{G_i}_x( X^i_{2S}\in A,{\tau}_{\gates_i}\le 2S) \right|\\
    &=\left|\sum_{s=0}^{2S} \sum_{y\in\gates_i}{\quench}^{G_i}_x(\bar{\tau}_{\gates_i}=s,\bar X_s=y)\left[{\quench}^{G_i}_{\muout}(\bar X_{{2S}-s}\in A)-\quench^{G_i}_y(X^i_{{2S}-s}\in A) \right]\right|.
\end{split}
\end{equation}
 Since $\bar X$ can be perfectly coupled, up to the first hitting of the set $\gates_i$, with the simple random walk on $G_i$, that in this proof we explicitly denote by $X^i$
 to avoid confusion, we have
 \begin{equation}\label{eq:mix2}
 	\begin{split}
 	\max_{x \in V_i}\left\|\bar{P}_i^{2S}(x,\cdot)-\bar\pi_i\right\|_{\rm TV}&\le\max_{x\in V_i}\|P_i^{2S}(x,\cdot)-\pi_i \|_{\rm TV}+\max_{x\in V_i}\max_{A\subset V_i}\mathcal{O}_{x}(A)+\op(1)\\
 	&=\max_{x\in V_i}\max_{A\subset V_i}\mathcal{O}_{x}(A)+\op(1)\,,
 	\end{split}
 \end{equation}
 where we used Corollary \ref{coro:mixing}.
It holds
\begin{equation}\label{eq:OxA}
 	\begin{split}
 	\mathcal{O}_{x}(A)\le \mathcal{O}^{\rm small}_x(A)+\mathcal{O}^{\rm large}_x(A)\,,
 \end{split}
 \end{equation}
 where
 \begin{equation*} \begin{split}
\mathcal{O}^{\rm small}_x(A) \coloneq \left|\sum_{s=0}^{S} \sum_{y\in\gates_i}{\quench}^{G_i}_x(\bar{\tau}_{\gates_i}=s,\bar X_s=y)\left[{\quench}^{G_i}_{\muout}(\bar X_{{2S}-s}\in A)-\quench^{G_i}_y(X^i_{{2S}-s}\in A) \right]\right|, \\
 	\mathcal{O}^{\rm large}_x(A)\coloneq\left|\sum_{s=S}^{2S} \sum_{y\in\gates_i}\quench^{G_i}_x(\bar{{\tau}}_{\gates_i}=s, \bar{X}_s=y)\left[{\quench}^{G_i}_{\muout}(\bar X_{{2S}-s}\in A)-\quench^{G_i}_y(X^i_{{2S}-s}\in A) \right]\right|.
     \end{split}
     \end{equation*}
 On the one hand, for what concerns $\mathcal{O}^{\rm large}_x(A)$,
 since $\bar{X}$ is coupled with $X^i$ up to time $\bar{\tau}_{\gates_i}$, bounding the quantity between square brackets by $2$, using Corollary \ref{coro:mixing} and Lemma \ref{lemma:pi-gates} it follows, for large $n$,
 \begin{equation}\label{eq:OxAlarge}
 	\begin{split}
 	\max_{x\in V_i}\max_{A\subset V_i}\mathcal{O}^{\rm large}_x(A)
    &\le 	2\max_{x\in V_i}\sum_{s=S}^{2S}\quench^{G_i}_x(X^i_s\in\gates_i)\\
 	&
 	\le 5{S}\pi_i(\gates_i)=\op(1)\,.
 	\end{split}
 \end{equation}
 On the other hand,  $\mathcal{O}^{\rm small}_x(A)$ can be bounded as follows: for any $A\subset V_i$
  \begin{equation}\label{eq:OxAsmall1}
 	\begin{split}
 		\max_{x\in V_i}\mathcal{O}^{\rm small}_x(A)&\le 	\max_{y\in\gates_i}\max_{s\le S}\left|{\quench}^{G_i}_{\muout}(\bar X_{{2S}-s}\in A)-\quench^{G_i}_y(X^i_{{2S}-s}\in A) \right|\\
 		&\le 	\max_{y\in\gates_i}\max_{s\le S}\left|{\quench}^{G_i}_{\muout}(\bar X_{{2S}-s}\in A)-\bar{\pi}_i(A)\right|+\\
 		&\qquad	+\max_{y\in\gates_i}\max_{s\le S}\left|{\quench}^{G_i}_y( X^i_{{2S}-s}\in A)-\pi_i(A) \right|+|\pi_i(A)-\bar{\pi}_i(A)|\\
 		&=	\max_{s\le S}\left|{\quench}^{G_i}_{\muout}(\bar X_{{2S}-s}\in A)-\bar{\pi}_i(A)\right|+\op(1)\,,
 	\end{split}
 \end{equation}
 where we added and subtracted $\pi_i(A) +\bar{\pi}_i(A)$, used twice the triangle inequality, Corollary \ref{coro:mixing} and \eqref{eq:pi-pibar}. Taking the maximum over $A\subset V_i$, we get
 \begin{equation}\label{eq:OxAsmall2}
 		\begin{split}
 		\max_{A\subset V_i}\max_{x\in V_i}\mathcal{O}^{\rm small}_x(A)&\le 
        \max_{s\le S}\left\|\muout\bar{P}_i^{{2S}-s}-\bar{\pi}_i\right\|_{\rm TV}+\op(1)\\
 		&\le \max_{s\le S}\left\|\muout  P_i^{2S-s}-\pi_i \right\|_{\rm TV}+\op(1)=\op(1)\,
 	\end{split}
 \end{equation}
 where for the last equality we simply applied Theorem \ref{thm:cutoff}, while for the second inequality, we added and subtracted $\muout P_i^{2S-s}-\pi_i$, used the triangle inequality, and used that,
 thanks to  Lemma \ref{lemma:key},
 \begin{equation}\label{eq:OxAsmall3}
 	\begin{split}
 		\max_{s\le S}\max_{A\subset V_i}\left|\sum_{y\in V_i}\muout(y)\left( \bar P_i^{2S-s}(y,A)-P_i^{2S-s}(y,A)\right) \right|&\le {\quench}^{G_i}_{\muout}(\bar{\tau}_{\gates_i}<{2S})=\op(1)\,.
 	\end{split}
 \end{equation}
At this point, the desired result, Eq.~\eqref{eq:est-T},
follows by putting together \eqref{eq:mix2}, \eqref{eq:OxA}, \eqref{eq:OxAlarge}, \eqref{eq:OxAsmall1}, \eqref{eq:OxAsmall2} and \eqref{eq:OxAsmall3}.
\end{proof}
	\begin{proof}[Proof of Lemma \ref{lemma:R-new}]
Follows at once by Lemmas \ref{lemma:T} and \ref{lemma:key} and the fact that, by Theorem \ref{thm:pi-char} and Proposition \ref{prop:concentration-degrees},
\begin{equation}
	\log\bigg(1/\min_{x\in \tilde V_i}\tilde{\pi}_i\bigg)=(1+\op (1))\log(n)\,.\qedhere
\end{equation}
\end{proof}

\begin{remark}\label{rmk:muout}
 In the following, it will be useful to consider the transition matrix $\hat P_i$ of a SRW on $V_i$, where the external out-edges of $\mathcal{G}_i$ (i.e, pointing towards another community) are canceled instead of being rewired. 
	Notice that the very same conclusion of Lemma \ref{lemma:key}, can be obtained replacing $P_i$, with $ \hat P_i$ to define the modified random walk $\bar X$. 
    Indeed, the only point at which the argument differs is the step (iv) in the definition of the annealed walk (right below \eqref{eq:517b}): instead of selecting a tail of $w\in\gates_i$ u.a.r., one has to select u.a.r.\ a \emph{non-marked} tail of $w\in\gates_i$. 
    Nevertheless, up to the first return to the set $\gates_i$, the two processes coincide. \\
    Notice also that the property of belonging to $\mathcal{M}$ in \eqref{eq:M} is stable under small perturbations. Then, Lemma \ref{lemma:key} holds for any approximation $\check{\mu}$ of $\mub$ (which is defined in \eqref{eq:5-26}) such that $\sup_{x \in \gates_i} \check{\mu}(x)/{\mub(x)}$ is \whp bounded. 
    Later, our choice of $\check \mu$ we will be the restriction to $\gates_i$ of the distribution $P_i^{s}(x,\cdot)$, for some $s \ge \log(n)^2$ and $x \in V_i$.
Indeed, Corollary \ref{coro:mixing} ensures that $P_i^{s}(x,\cdot)$ can be \whp approximated in $\ell^\infty$-norm by $\pi_i$, and then, on this event, its restriction to $\gates_i$ is near to $\mub$ in $\ell^\infty$-norm.
\end{remark}
\subsection{The coupling}\label{suse:coupling}
For an arbitrary initial state $x\in V_i$, we will now consider a coupled construction of the simple random walk on $G$ started at $x$, $(X_t)_{t\ge 0}$, and a toy process $(Y_t)_{t\ge0}$, which is related to the quasi-stationary behavior of the random walk out of $\gates_i$.
\begin{definition}\label{def:coupling}
We consider the joint Markov processes $(X_t,Y_t)_{t\ge0}$:
\begin{enumerate}
	\item On the one hand, marginally $(X_t)_{t\ge 0}$ is the simple random walk on $G$ started at some $x\in V_i$;
	\item on the other hand, marginally $(Y_t)_{t\ge 0}$ is a random walk on $G_i$, started at the quasi-stationary distribution $\muqs$, such that, when it hits $\gates_i$, at the next step $Y$ is reinitialized at  $\muqs P_i$.  
	Moreover, we enrich such a process by appending a mark, $(\kappa_t,\rho_t)$, defined as follows: we start by setting $(\kappa_0,\rho_0)=(0,0)$ and, whenever at some time $t>0$, $Y_t=v$, for some $v \in\gates_i$, we set $\kappa_t=\kappa_{t-1}+1$ and toss a coin with probability of success $O_v^+/D_v^+$: if it is a success, we set $\rho_t=\rho_{t-1}+1$, otherwise $\rho_t=\rho_{t-1}$.
\end{enumerate}
We couple the two processes as follows: the coupling is articulated in \emph{stages}, and each stage is made by two steps each, (A) and (possibly) (B), where
\begin{itemize}
	\item[(A)] we use the optimal coupling between $P_i^t(x,\cdot)$ and $\muqs P_i^t(\cdot)$ up to the first time $t$ such that:
	\begin{itemize}
		\item[(i)] either $X_t=Y_t$; in this case, we then let the two walks evolve following the same path up to the hitting time of the set $\gates_i$.
		\item[(ii)] or $X_t\neq Y_t$ and $Y_t\in \gates_i$; in this case we declare the coupling as {\em failed} and continue the construction of the two processes independently.
		\item[(iii)] or $X_{t}\neq Y_{t}$, and $X_t$ traverses a rewired edge; in this case, we declare the coupling as {\em failed} and continue the construction of the two processes independently.
	\end{itemize}
	\item[(B)] In case (i), call $v\in\gates_i$ the vertex visited by the two processes and:
	\begin{itemize}
		\item[(iv)] Toss a coin with success probability $O_v^+/D_v^+$:
		\begin{itemize}
			\item if it results in a head,
			we say that the random walks $X$ traverses a rewired edge of $v$ u.a.r., declare the coupling as \emph{successful} and continue the construction of the two processes independently;
			\item  if it results in a tail, then let the random walk $X$ choose one of the internal out-edges of $v$ u.a.r., and let the random walk $Y$ choose a new starting point according to $\muqs P_i$.
			After that, repeat stage (A).
		\end{itemize}
	\end{itemize}
\end{itemize}
We will call $\hat\quench^{G_i}_{x,\muqs}$ the law of the coupling just described. 
\end{definition}
\begin{proposition}\label{prop:coupling}
	Let $\mathcal{S}$ denote the event that the coupling of law $\hat\quench^{G_i}_{x,\muqs}$ is successful. Then
	\begin{equation}
		\max_{i \le m}\max_{x\in V_i}\hat\quench^{G_i}_{x,\muqs}(\mathcal{S})=1-\op(1)\,.
	\end{equation}
\end{proposition}
Before presenting the proof, let us point out that on the event $\mathcal{S}$, the time $\tauj$ at which the random walk does its first inter-community jump coincides with a much simpler random time, namely 
\begin{equation}\label{eq:def-tau-rho}
	\tau_\rho=\inf\{t\ge 0\mid \rho_t>0\}\,.
\end{equation}

Thanks to the fact that $(Y_t)_{t \ge 0}$ is  reinitialized to $\muqs P_i$ any time it hits $\gates_i$, and since $\muqs P_i(\gates_{i})=\mathfrak{l}_i$,
we have that $(\kappa_t-\kappa_{t-1})_{t\ge 1}$ are \iid Bernoulli random variables with parameter $\mathfrak{l}_i$. Therefore,
 the quantities
\begin{equation}\label{eq:def-sigma-k}
	\sigma_k=\inf\{t\ge 0\mid \kappa_t=k\}-\inf\{t\ge 0\mid \kappa_t={k-1}\}\,,\qquad k\ge 1\,,
\end{equation}
are \iid geometrically distributed with parameter $\mathfrak{l}_i$. 
Since, thanks to Proposition \ref{prop:rate}, $\mathfrak{l}_i=(1+\op(1))\alpha\lambda\log(n)$, as an immediate corollary, we have the following estimate, which will be useful later.
\begin{corollary}\label{coro:sigma}
	For every $k\ge 0$ it holds
	\begin{equation}\label{eq:sigma}
		\max_{i \le m}\hat{\quench}^{G_i}_{\muqs}\left(\sigma_k<n^{\frac13} \right)\le n^{-\frac17} \,.
	\end{equation}

\end{corollary}
\begin{proof} It holds
    $$\hat{\quench}^{G_i}_{\muqs}\left(\sigma_k<n^{\frac 13} \right)=1-(1-\mathfrak{l}_i)^{n^{\frac 13}}\sim 1-\exp{(-n^{\frac13}\alpha\lambda \log n)}.$$
    Since $\alpha \ll n^{-1/2}\log^2(n)$, then the latter decays as $n^{\frac 13}\alpha \log n \ll n^{-\frac17}$.
\end{proof}
Notice that if in stage (B) we had thrown, at each time $t$ that $\kappa_t-\kappa_{t-1}=1$, a coin with success probability $(\lambda\log(n))^{-1}$, then we could immediately deduce that the time of the first success is asymptotically exponentially distributed with rate $\alpha$. Of course, for the process $(Y_t)_{t\ge 0}$, ''the probability of success`` is not $(\lambda\log(n))^{-1}$: it depends on the gate visited by the process (through its out-degree and the number of its rewired out-edges). The next proposition shows that, thanks to a sort of \emph{homogenization property}, the toy model just described  actually captures the correct picture. 
\begin{proposition}\label{prop:renewal}
	Recalling the definition of $\tau_{\rho}$ in \eqref{eq:def-tau-rho}
	\begin{equation}\label{eq:tau-rho}
		\max_{i \le m}\sup_{t \ge 0}\left|\hat{\quench}^{G_i}_{\muqs}(\alpha \tau_\rho> t)-{\rm e}^{-t}\right|=\op(1)\,.
	\end{equation}
\end{proposition}
Being the coupling \whp successful by Proposition \ref{prop:coupling}, as a consequence of Proposition \ref{prop:renewal} we deduce the following.
\begin{corollary}\label{coro:renewal}
	Recalling the definition of \eqref{eq:def-tau-jump}, 
	\begin{equation}\label{eq:tau-rho-2}
		\max_{i \le m}\max_{x\in V_i}\sup_{t \ge 0}\left|\mathbf{P}^G_x(\alpha \tauj> t)-{\rm e}^{-t}\right|=\op(1)\,.
	\end{equation}
\end{corollary}
\begin{proof}[Proof of Proposition \ref{prop:renewal}]
	Fix $i \le m$. The proof is articulated in five steps:
	\begin{enumerate}
		\item[{\bf [1]}] First, we show that the density of gates in $\gates_i$ which have a out-degree different (at first order) than $\lambda\log(n)$ is at most $\log(n)^{-2}$;
		\item[{\bf [2]}] Then we show that the same is true for the gates $x \in \gates_i$ which have $O^+_x\ge 2$;
		\item[{\bf [3]}] Then we show that the distribution of the first element of $\gates_i$ visited by the random walk initialized at $\muqs$ is essentially uniform;
		\item[{\bf [4]}] Call \emph{nice} the subset of gates which do not have the properties in step {\bf [1]} and {\bf [2]}. We show that the first $\log(n)^{3/2}$ visits to $\gates_i$ occur at \emph{nice} gates;
		\item[{\bf [5]}] We wrap up the argument developed in the previous steps to conclude the validity of \eqref{eq:tau-rho}.\\
	\end{enumerate}
    
{\bf Step [1].} Call 
	$$\cE_\varepsilon=\left\{ \,\#\!\left\{x\in \gates_i : \left|D_x^+-\lambda\log(n)\right| \ge \varepsilon\lambda\log(n) \right\}>|\gates_i|\log(n)^{-2} \right\}\,.$$
	As mentioned, we want to show that $\p(\cE_\varepsilon)=o(1)$, for some $\eps=o(1)$. Notice the distribution of the out-degree of $x$ conditioned on the event $\{x\in\gates_i\}$ can be written explicitly as follows: for any $j\in[1,n]$
	\begin{equation} \begin{split}
		\p(D_x^+=j|x \in \gates_i)&
		=\frac{\p(x \in \gates_i|D_x^+=j)\p(D_x^+=j)}{\p(x \in \gates_i)}\\
		&=\frac{(1-(1-\alpha)^j)\p(D_x^+=j)}{\sum_{k=1}^n(1-(1-\alpha)^k)\p(D_x^+=k)}\\
		&=\frac{(1-(1-\alpha)^j)\p(D_x^+=j)}{1-(1-\alpha p)^{n}}
		\le 2\p(D_x^+=j)\,.
	\end{split}
	\end{equation}
	Since by the Chernoff bound it holds
	\begin{equation}
		\p\big(|D_x^+-\lambda\log(n)|>\varepsilon\log(n)\big)\le \exp\left(-\tfrac13\varepsilon^2\lambda\log(n)\right)\,,
	\end{equation}
	we get
	\begin{equation}
		\p\big(|D_x^+-\lambda\log(n)|>\varepsilon\log(n)\mid x\in\gates_i\big)\le 2\exp\left(-\tfrac13\varepsilon^2\lambda\log(n)\right)\,.
	\end{equation}
	Now, notice that, thanks to Eq.~\eqref{one},
	\begin{equation}
		\p(\cE_\varepsilon)=\p\big(\cE_\varepsilon\cap\{|\gates_i|/(n\alpha\lambda\log(n))\in[1/2,2]\}\big)+o(1)\,.
	\end{equation}
	On the event $\{|\gates_i|/(n\alpha\lambda\log(n))\in[1/2,2]\}$, we can use the union bound and the fact that $\log\binom{N}{m}=m \log(\frac{N}{m})(1+o(1))\le 2m\log(\frac{N}{m}),$ for $m \ll N$, to get
	\begin{equation}
		\begin{split}
			\p\big(\cE_\varepsilon&\cap\{|\gates_i|/(n\alpha\lambda\log(n))\in[1/2,2]\}\big)\\
			& \le \max_{C \in \left[\frac{1}{2},2\right]}\binom{2n\alpha\lambda\log(n)}{C n\alpha\lambda\log(n)^{-1}}\left(2\exp\left(-\tfrac13 \varepsilon^2\lambda \log(n)\right)\right)^{C n\alpha \lambda\log(n)^{-1}} \\
             & \lesssim \max_{C \in \left[\frac{1}{2},2\right]}\exp{\left(2\frac{C n\alpha\lambda}{\log(n)} \log \left(\frac2C\log(n)^2\right)\right)}
           \left(2\exp\left(-\tfrac13 \varepsilon^2\lambda \log(n)\right)\right)^{\frac{C n\alpha \lambda}{\log(n)}} \\
            & \lesssim \max_{C \in \left[\frac{1}{2},2\right]}\exp{\left(2C n\alpha\lambda\frac{\log(\log(n)^2)}{\log(n)}\right)}
            2^{\frac{C n\alpha \lambda}{\log(n)}}\exp\left(-\tfrac C3 \varepsilon^2\lambda^2 n\alpha\right) \\
            &\le \exp\left(n\alpha\lambda\left(8\frac{\log\log(n)}{\log(n)}+{\frac{2 \log(2) }{\log(n)}}-\varepsilon^2\frac 6\lambda\right)\right)\,.
		\end{split}
	\end{equation} 
    Therefore, we can choose, for example, $\varepsilon=\log(n)^{-1/3}$, and have
	\begin{equation}\label{eq:est-E}
		\p(\cE_\varepsilon)=o(1)\,.
	\end{equation}
    
{\bf Step [2].}   Call
    $
		\gates_{i,\rm bad}=\left\{x \in \gates_i \mid O_x^+ \ge 2\right\},
    $
	and consider the event
	\begin{equation}
		\mathcal{R}=\left\{|\gates_{i,\rm bad}|>|\gates_i|\log(n)^{-2} \right\}\,.
	\end{equation}
	To show that the latter has a vanishingly small probability it is enough to realize that

	\begin{equation}
		\E[|\gates_{i,\rm bad}|]\le n\alpha^2\lambda^2\log(n)^2\,.
	\end{equation}
	Indeed, by Markov inequality, and recalling $\alpha \ll n^{-1/2} \log(n)^2$,
	\begin{equation}\label{eq:est-D}
		\begin{split}
			\p\left(\mathcal{R} \right)&=\p\left(\mathcal{R}\cap \{|\gates_i|\ge \tfrac12\E[|\gates_i|]\} \right)+o(1)\\
			&\le \frac{2n\alpha^2\lambda^2\log(n)^4}{n\alpha\lambda\log(n)}+o(1)\conv 0\,.
		\end{split}
	\end{equation}
    
{\bf Step [3].} Now we focus on the probability distribution
on the set of gates, according to which the first gate is visited, when starting at quasi-stationarity. We define
	\begin{equation} \label{eq:muinb}
		\begin{split}
			\muinb(x)&=\frac{\sum_{y\in V_i\setminus \gates_i}\muqs(y)P(y,x)}{\sum_{y\in V_i\setminus \gates_i}\muqs(y)\sum_{z\in \gates_i}P(y,z)}\\
			&=(1+\op(1))\frac{\sum_{y\in V_i\setminus \gates_i}\muqs(y)P(y,x)}{\alpha\lambda\log(n)} \,,\qquad x\in\gates_i\,,
		\end{split}
	\end{equation} 
	where the second line is due to $\mu_i^\star P(\gates_i)=\mathfrak{l}_i$ and the $\op(1)$ term is uniform over $x\in\gates_i$.\\
	We claim that
	\begin{align}
		\max_{x\not\in \gates_i}\frac{\muqs(x)}{\pi_i(x)}\le 1+\op(1)\,.
	\end{align}
	To prove it, recalling the definition of $S$ in Corollary \ref{coro:mixing}, we follow the same argument as in  \cite[Eqs.~5.15--5.19]{MQS}. We have that, for any $x\not\in\gates_{i}$,
	\begin{equation} \begin{split}\label{eq:eq:MQS}
		\muqs(x)&=(1-\mathfrak{l}_i)^{-S}\sum_{y\in V_i\setminus\gates_{i}}\muqs(y)[P_i]_{\gates_{i}}^S(y,x)\\
		&\le(1-\mathfrak{l}_i)^{-S}\sum_{y\in V_i\setminus\gates_{i}}\muqs(y)P_i^S(y,x)\\
		&=(1+\op(1))(1-\mathfrak{l}_i)^{-S}\sum_{y\in V_i\setminus\gates_{i}}\muqs(y)\pi_i(x)\\
		&=(1+\op(1))(1-\mathfrak{l}_i)^{-S}\pi_{i}(x)\\
		&=(1+\op(1))\pi_{i}(x)\,, \end{split}
	\end{equation}
	where the first line follows by the definition of quasi-stationary distribution; the second line follows by the trivial fact $[P_i]_{\gates_{i}}^S(y,x) \le P_i^S(y,x) $ for any $x,y\not\in\gates_{i}$; the third line follows by Corollary \ref{coro:mixing} and the fact that, thanks to Theorem \ref{thm:pi-char} and Proposition \ref{prop:concentration-degrees}, $\pi_i(x)+\op(1)=(1+\op(1))\pi_i(x)$ uniformly in $x\in V_i$; in the fourth line we simply took the sum over $y\in V_i\setminus\gates_{i}$; and in the last line we used that, thanks to Eq.~\eqref{eq:rate}, $\mathfrak{l}_iS=\op(1)$.
	Therefore, 
	\begin{equation}\label{eq:unif-muqs}
		\begin{split}
			\max_{x\in\gates_i}\muinb(x)&\le (1+\op(1))	\max_{x\in\gates_i}\frac{\sum_{y\in V_i\setminus \gates_i}\pi_i(y)P(y,x)}{\alpha\lambda\log(n)}\\
			&\le (1+\op(1))	\max_{x\in\gates_i}\frac{\pi_i(x)}{\alpha\lambda\log(n)}
			=O_{\p}\left(\frac{1}{|\gates_i|}\right)\,,
		\end{split}
	\end{equation}
	where the last bound follows from Theorem \ref{thm:pi-char} and Lemma \ref{lemma:pi-gates}. In other words, the gates visited by the process are essentially uniformly distributed.\\
    
{\bf Step [4].}  Fix $\varepsilon=\log(n)^{-1/3}$ and call
	\begin{equation}
		\gates_{i,\rm nice}=\left\{x\in\gates_i\setminus\gates_{i,\rm bad}\mid D_x^+\in[(1-\varepsilon)\lambda\log(n),(1+\varepsilon)\lambda\log(n)] \right\}\,.
	\end{equation}
	Notice that, thanks to steps {\bf [1]} and {\bf [2]}, in particular Eqs.~\eqref{eq:est-D} and \eqref{eq:est-E} respectively,
	\begin{equation}\label{eq:est-DE}
		\p\left(\frac{|\gates_{i,\rm nice}|}{|\gates_{i}|}\ge 1-2\log(n)^{-2}\right)=1-o(1)\,.
	\end{equation}
	Call $(B_k)_{k\ge 1}$ the  the sequence of vertices in $\gates_i$ that are visited by the process $(Y_t)_{t\ge 0}$. 
    Consider the event
    $$\mathcal{W} = \{ B_k \in\gates_{i,\rm nice},\, \forall k \le \log(n)^{3/2} \}.$$
	From step \textbf{[3]}, particularly \eqref{eq:est-DE} and \eqref{eq:unif-muqs}, it follows
	\begin{equation}\label{eq:W}
		\hat{\quench}^{G_i}_{\muqs}\big(\mathcal{W} \big)=1-\op(1)\,.
	\end{equation}

{\bf Step [5].} Now recall the notation $(\kappa_t,\rho_t)_{t\ge 0}$ given in Definition \ref{def:coupling} and $\tau_\rho$ as in Eq.~\eqref{eq:def-tau-rho}. 
For $s \ge 0$ it holds
    \begin{equation}\label{eq:split-ub} 
		\hat{\quench}^{G_i}_{\muqs}(\tau_\rho> s)=\hat{\quench}^{G_i}_{\muqs}(\rho_s=0)= \hat{\quench}^{G_i}_{\muqs}(\rho_s=0,\mathcal{W})+	\hat{\quench}^{G_i}_{\muqs}(\rho_s=0,\mathcal{W}^\complement)\,.
	\end{equation}
	The second term on the \rhs of \eqref{eq:split-ub} is $\op(1)$ thanks to \eqref{eq:W}. Let us bound the other one.
    Recalling that $(\kappa_s-\kappa_{s-1})_{s\ge 1}$ are \iid Bernoulli random variables with parameter $\mathfrak{l}_i$, it immediately follows that $\kappa_s\sim \bin(s,\mathfrak{l}_i)$,
    so that, as $\mathfrak{l}_i=\lambda\alpha\log(n)(1+\op(1))=\op(1)$, it holds $$\hat{\mathbf{Var}}^{G_i}_{\muqs}(\kappa_s)
    =s(\mathfrak{l}_i-\mathfrak{l}_i^2)
    =s\lambda\alpha\log(n)(1+\op(1))
    =  \hat{\Equench}^{G_i}_{\muqs}[\kappa_s](1+\op(1)).$$
    Taking $\delta=(\log\log\log(n))^{-1}$ and 
    $s = t/\alpha$, where $t \in \left[ C_n^{-1},{C_n}\right] $ and $C_n$ is a sequence diverging sufficiently slowly,
    by Chebychev inequality, 
	\begin{equation}
    \begin{split}
    \hat{\quench}^{G_i}_{\muqs}\left(\left|\frac{\kappa_s}{{s}{\alpha\lambda\log(n)}}-1\right|\ge \delta\right)
    &\le\frac{\hat{\Equench}^{G_i}_{\muqs}[\kappa_s]}{\delta^2 (\hat{\Equench}^{G_i}_{\muqs}[\kappa_s])^2}(1+\op(1))\\
    &\le \frac{(\log\log\log(n))^2C_n}{\lambda\log(n)}(1+\op(1))=\op(1)\,.
    \end{split}
    \end{equation}
    In conclusion, the first term on the \rhs of \eqref{eq:split-ub} can be bounded, for $s$ as above, by
\begin{equation}
		\begin{split}
			\hat{\quench}^{G_i}_{\muqs}(\rho_s=0,\mathcal{W})
			&= (1-\op(1))\hat{\quench}^{G_i}_{\muqs}\left(\rho_s=0\,\bigg\vert\,\mathcal{W}\,\cap\left\{ \left|\frac{\kappa_s}{{s}{\alpha\lambda\log(n)}}-1\right|<\delta\right\} \right)+\op(1)\\
			&=(1-\op(1))\left(1-\frac{1}{(1+O(\varepsilon))\lambda\log(n)}\right)^{(1+O(\delta))s\alpha\lambda\log(n)}+\op(1)\\
			&={\rm e}^{-s\alpha}+\op(1)\,,
		\end{split} 
	\end{equation}
    where we used that $\eps$ and $\delta$ are vanishing and, in the second line, we used the conditioning and that, for 
    $s \alpha \log(n) 
    \ll \log(n)^{3/2}$, on the event $\mathcal{W}$, the gates visited by the random walk have degrees in $[(1-\eps)\lambda\log(n),(1+\eps)\lambda \log(n)]$.
	Then, provided that $C_n$ diverges sufficiently slowly, 
        recalling $s\alpha=t$, we get
\begin{equation}
	 	\max_{t\in\left[C_n^{-1}, {C_n}\right]}
         \left|\hat{\quench}^{G_i}_{\muqs}\left(\tau_\rho> \frac{t}\alpha\right)- {\rm e}^{-t}\right|=\op(1)\,,
\end{equation}
and this concludes the proof.
\end{proof}

\newcommand{\muinh}{\hat{\mu}^{\rm in}_{\gates_i}}

We are now in a good position to present the proof of Proposition \ref{prop:coupling}.

\begin{proof}[Proof of Proposition \ref{prop:coupling}]
	First of all we observe the following fact: if $\rho_t>0$ for some $t\ge0$ then, at time $t$, we have sufficient information to declare if the coupling is \emph{successful} or \emph{failed}. Therefore, thanks to Proposition \ref{prop:renewal}, with probability $1-\op(1)$  the coupling consists of less than $\log(n)^{3/2}$ stages. Moreover, thanks to Theorem \ref{thm:cutoff} and the subadditivity of the total variation distance, the probability of a meeting before time $\log(n)^2$ is $1-\op(n^{-1})$ by Corollary \ref{coro:mixing}. This means that the probability that along the coupling there is a stage in which the two processes meet after time $\log(n)^2$ is $\op(1)$.
	
	Now, each stage the coupling might fail because of two alternative reasons:
	\begin{enumerate}
		\item 
        On the one hand, a stage might produce a failure if (ii) in Definition \ref{def:coupling} happens: the process $Y$ hits $\gates_i$ before meeting the first process under optimal coupling. By Corollary \ref{coro:sigma} the probability that $Y$ hits $\gates_i$ before time $\log(n)^2$ is $\op (n^{-\frac17})$. 
        Hence, the probability that the coupling fails due to this kind of event is $\op (\log(n)^{3/2}n^{-\frac17})$.
		\item 
        On the other hand, a stage might produce a failure if (iii) in Definition \ref{def:coupling} happens: the process $X$ might visit $\gates_i$ and traverse a rewired edge before time $\log(n)^2$.
         For what concerns the first stage, in which the starting point of the walk is arbitrary, thanks to Proposition \ref{prop:jump-time}, the probability of this event is $\op(1)$. 
For any successive stage, first notice that after time $\log(n)^2$, by Corollary \ref{coro:mixing}, the distribution of the random walk $X$ can be \whp approximated in $\ell^\infty$-norm by $\pi_i$. 
        Thanks to Lemma \ref{lemma:key}, complemented with Remark \ref{rmk:muout}, we have that the expected number of visits to $\gates_i$ within time $\log(n)^2$ is $o_{\p}(n^{-1/3})$. 
		Hence, the probability that there exists a stage in which $\gates_i$ is visited before $\log(n)^2$ is $o_{\p}(\log(n)^{3/2}n^{-\frac13})$.
	\end{enumerate}
\end{proof}

\subsection{Proof of Proposition \ref{prop:quenched-law-new} and Theorem \ref{thm:pi-approximation-new}}\label{suse:proof-prop-thm}
We are almost ready to conclude the proof of Proposition \ref{prop:quenched-law-new}, and then to show how Theorem \ref{thm:pi-approximation-new} 
can easily be deduced from it using the same set of arguments used to deduce Theorem \ref{thm:pi-approximation} from Proposition \ref{prop:quenched-law}.

	We start by stressing that, thanks to Corollary \ref{coro:renewal}, we know that the jumping time of the process starting at any $x\in V$, properly scaled, is well approximated by a standard exponential random variable. We now want to show that the jumps occur uniformly at random among communities. 
\begin{lemma}\label{lemma:unif-jumps} 	 For $\sqrt{n}\log(n)^{-2}\ll \alpha^{-1}\ll \lambda n \log(n)$ 
	\begin{equation}
		\max_{i \le m}\max_{j\neq i}\max_{x\in V_i}\left|\quench^G_x(X_{\tauj}\in V_j)- \frac1{m-1}\right|=\op (1)\,.
	\end{equation}
\end{lemma}
\begin{proof}
	Notice that vertices in $\gates_{i,\rm nice}$ have, by definition, a unique rewired edge.
To each $y \in \gates_{i,\rm nice}$, one can assign a mark $J(y)$, chosen
u.a.r.~in $[m]\setminus\{i\}$, representing the community to which the unique rewired edge of $y$ connects after rewiring.
    We will then partition
	\begin{equation}
		\gates_{i,\rm nice}=\sqcup_{j\neq i}\gates_{i,\rm nice}^j\,,
	\end{equation}
	where $\gates_{i,\rm nice}^j$ is the subset of $\gates_{i,\rm nice}$ having mark $j$.
	Since $|\gates_{i,\rm nice}|=\omega_{\p}(1)$, and by the uniform choice,
	\begin{equation}\label{eq:even}
		\max_{i \le m}\max_{j\neq i}\left|\frac{|\gates_{i,\rm nice}^j|/|\gates_{i,\rm nice}|}{(m-1)^{-1}} -1\right|=\op(1)\,.
	\end{equation}
	Moreover, specializing \eqref{eq:unif-muqs} to the case $x\in\gates_{i,\rm nice}$, where $D_x^+$ can be well estimated, one has
	\begin{equation}\label{eq:nice}
		\max_{x\in\gates_{i,\rm nice}}\muinb(x)\le (1+\op(1))\frac{1}{|\gates_i|}=(1+\op(1))\frac{1}{|\gates_{i,\rm nice}|}\,,
	\end{equation}
	where in the last step we used \eqref{eq:est-DE}. 
    As a consequence, the hitting measure of $\gates_i$ is asymptotically uniform on $\gates_{i,\rm nice}$. In conclusion, the desired result follows by putting together Proposition \ref{prop:coupling}, \eqref{eq:even} and \eqref{eq:nice}.
\end{proof}
Thanks to Lemma \ref{lemma:unif-jumps} we can couple the evolution of the non-Markovian process $(c(X_t))_{t\ge 0}$ with the evolution of a simple random walk on a complete graph with $m$ vertices, with transitions as in \eqref{eq:trans-complete}. Using this fact,
we can finally prove Proposition \ref{prop:quenched-law-new} and Theorem \ref{thm:pi-approximation-new}.

\begin{proof}[Proof Proposition \ref{prop:quenched-law-new}]
	We consider the following iterated version of the coupling in Definition \ref{def:coupling}:
	\begin{itemize}
		\item We start the coupling at some $x\in V_i$ for some $i \le m$.
		\item If the coupling succeeds and at time $\tauj$ the process is found at some $y\in V_j$ with $j\neq i$, then the coupling is restarted with the initial distribution $\delta_y\otimes \mu^\star_j$.
		\item Iterate this procedure up to the first iteration at which the corresponding coupling fails.
	\end{itemize}
Call $\widehat{\quench}^G_x$ the law of this coupling, 
and fix an arbitrary sequence of integers $C\equiv C_n\gg 1$.
Call $\mathcal{E}_C^{\rm succ}$ the event in which the coupling succeeds up to the $C^2$-th iteration. Then, thanks to Proposition \ref{prop:coupling}
\begin{equation}\label{eq:582}
\min_{x\in V}\widehat{\quench}^G_x(\mathcal{E}_C^{\rm succ})=(1-\op(1))^{C^2}\,.
\end{equation}
Call now $\mathcal{E}_C^{\rm iter}$ the event in which by time $t=C\alpha^{-1}$ there are at most $C^2$ iterations. Then, 
\begin{equation} \label{eq:5.88}
\begin{split}
      \min_{x\in V}\widehat \quench^G_x(\mathcal{E}_C^{\rm iter}\cap\mathcal{E}_C^{\rm succ} )
       &=\min_{x\in V}\widehat\quench^G_x(\mathcal{E}_C^{\rm succ} )\widehat\quench^G_x(\mathcal{E}_C^{\rm iter}\mid\mathcal{E}_C^{\rm succ} )=(1-\op(1))^{2C^2}\,,
    \end{split}
\end{equation}
where in the second equality we used \eqref{eq:582} and Corollary \ref{coro:renewal} on a union bound on the $C^2$ iterations of the coupling, which ensures that---conditionally on the event $\mathcal{E}_C^{\rm succ}$---the length of each iteration is asymptotically geometrically distributed with parameter $\alpha$. As a consequence, the $\op(1)$ in \eqref{eq:5.88}, is the maximal one among the errors in \eqref{eq:582} and Corollary \ref{coro:renewal}. In conclusion, as soon as $C\gg 1$ is chosen to diverge sufficiently slowly, the iterative version of the coupling will be successful \whp up to time $t=C\alpha^{-1}$. In particular, up to that time, the number of iterations will be \whp at most $C^2$ and the length of  such iterations will be \whp coupled to an independent geometric random variable of rate $\alpha$.
Moreover, thanks to  Lemma \ref{lemma:unif-jumps}, the inter-community jump ending each iteration is approximately uniformly distributed among the other communities. Therefore, if $(\tauj^{({\ell})})_{{\ell} \le C^2}$ denotes the sequence of the lengths of the first $C^2$ iterations, we can couple the sequence $\big(c(X_{\tauj^{({\ell})}})\big)_{\ell\le C^2}$ with the trajectory (of length $C^2$) of the process with transition matrix $Q$ at a total-variation cost bounded by $1-C^2\op(1)$. The latter goes to $1$ as soon as $C$ diverges sufficiently slowly. At this point the desired result follows at once by noting that if, for any $t\ge0$, we sample \iid geometric random variables of parameter $\alpha$ up to the first time in which their cumulative sum is above $t$ (hence, sampling the jump times), and then we sample the path of a simple random walk on the complete graph, we can entirely reconstruct the trajectory of process with transition matrix $Q$.

\end{proof}

\begin{proof}[Proof of Theorem \ref{thm:pi-approximation-new}]
The proof of this fact follows the same line of argument as in Theorem~\ref{thm:pi-approximation}. 
In our setting, Eq.~\eqref{1/m} follows from Proposition~\ref{prop:quenched-law-new} 
for any $T \ll C\alpha^{-1}$, where $C \gg 1$ is as given
in Proposition~\ref{prop:quenched-law-new}.
The remainder of the argument proceeds through the same steps, leading to the validity of Eq.~\eqref{eq:pi-approx-1} 
for any $\alpha^{-1} \ll T \ll \sqrt{C}\alpha^{-1}$. Finally, by the monotonicity of the distance to equilibrium, 
the result holds for any $T \gg \alpha^{-1}$.
\end{proof}

\section{Critical and subcritical regime}\label{sect:entropic-method}
In this section, we analyze the mixing behavior of the random walk in the regime $\alpha^{-1} \lesssim \tent$. 
This range of $\alpha$ includes both the subcritical and critical cases, described in Eqs.~\eqref{eq:sub-cutoff} 
and \eqref{eq:crit-cutoffmeta}.
The proofs of these results are adaptations of the techniques developed in \cite{BCS1,BP} to the current setting 
with multiple communities. 
Since it will be useful later, in what follows we take $\alpha^{-1}=C\tent$, for $C>0$. We again stress that Eq.~\eqref{eq:sub-cutoff} will hold only for $C\to 0$.
\subsection{Concentration results for random walk paths}
The first random object to study is given by the following definition. For each oriented path $\pat=(x_0,\dots,x_t)$ in $G$, we define the quenched probability \textit{mass} of $\pat$ as
$$\mass(\pat)\coloneqq\prod_{s = 0}^{t-1} P(x_s,x_{s+1})\,.$$
We can state the following result, which can be written in the shape of a Law of Large Numbers for the variables $\log(P(X_s,X_{s+1}))$. 

\begin{theorem}[Quenched LLN] \label{thm:LLN}
	Let $t=\Theta(\tent)$ and $\theta \in (0,1)$ be such that there exists $\rho>0$, $\rho\neq1$, satisfying $\log \theta = \rho \entropy t (1+o(1))$.
		Then
	\begin{equation}
		\max_{x \in V} \left|\quenchG_x\left(\mass(X_0,X_1,\dots,X_t)>\theta\right)-\one_{\{\rho>1\}} \right|\, \pconv 0.
	\end{equation}
		\end{theorem}
	\begin{proof}
		The proof follows by a straightforward adaptation of the proof of \cite[Theorem 3]{BP} which, in turn, is an adaptation of the proof of \cite[Theorem 4]{BCS2}. 
		Indeed, all vertex out-degrees have the same law, $\bin(n,p)$, and the proof in the above mentioned references does not rely on the details of the graph structure but only on the out-degree distribution and the fact that the random walk is \whp non-backtracking on the timescale $\log(n)$.
			\end{proof}
	This result has the same shape of Eq.~\eqref{eq:sub-cutoff} and is the core of the cutoff result. In fact, it is easy to observe that $\theta={\rm e}^{-\entropy \tent}=n^{-1}$ provides a threshold for the concentration of the mass of the paths with length $\tent$. That specific choice of $\theta$ is not covered by the theorem above, but in \cite{BP} it is possible to find a quenched CLT refinement, valid in this critical regime, to describe the cutoff window under a suitable hypothesis.\\
	
	The second ingredient is provided by the following family of paths, which allows us to capture typical features of SRW paths.
	\begin{definition}[Nice paths] \label{def:nice-paths}
		Let $\eps \in (0,1)$, and
		\begin{equation} \label{eq:he}
			\h\coloneqq 2\eps \tent, \quad \quad \s\coloneqq(1-\eps)\tent, \quad \quad t\coloneqq\s+\h+1=(1+\eps)\tent+1.
		\end{equation}
		We say that a path $\pat=(x,x_1,x_2,\dots,x_{t-1},y)$ of length $t$ from $x$ to $y$ is \emph{nice} if:
		\begin{enumerate} 
			\item[(i)] the entire path is such that $\mass(\pat)\le \frac{1}{n \log^3 n}$;
			\item[(ii)] the sub-path $(x_{\s+1},\dots,x_{t-1},y)$ is the unique path in $G$ of length at most $\h$ from $x_{\s}$ to $y$; 
			\item[(iii)] the sub-path $(x,x_1,\dots,x_{\s})$ is contained in the random tree $\mathcal{T}_x(\s)$ constructed as follows:\vspace{0.2cm} \\
			\emph{Fix a realization of $G$ and a root node $x \in V$.
			Let $\mathcal{G}^0=\mathcal{T}^0:=\{x\}$.  Then, for $\ell \ge 1$:
			\begin{enumerate}
				\item[(1)] Let $\mathcal{E}^{\ell}\coloneq\{(v,y) : v \in \mathcal{G}^{\ell-1}, y \in \mathcal{B}^+_v(1)\setminus \mathcal{G}^{\ell-1}  \}$, 
				be the set of edges which have not been visited by the first $\ell -1$ 
				iterations,  and with tails in $\mathcal{G}^{\ell-1}$. 
				\item[(2)] Choose $e=(v,y)\in \mathcal{E}^{\ell}$ such that, if $\pat_{x,v}$ is the unique path from $x$ to $v$ in $\mathcal{G}^{\ell-1}$,
					$\mass(\pat_{x,v}) (D^+_{v})^{-1}$ is not below ${\rm e}^{-(1+\eps)\entropy \s}$ and it is maximal among $(v,y) \in \mathcal{E}^{\ell}$,
					and $v$ is at distance at most $\s - 1$ from $x$
				(use a deterministic criterion to break ties).\\
				If such edge does not exist, stop the procedure and set $\mathcal{T}_x(\s)\coloneqq\mathcal{T}^{\ell-1}$; 
				\item[(3)] Generate $\mathcal{G}^\ell$ by adding $e$ to 
				$\mathcal{G}^{\ell-1}$ ;
				\item[(4)] Generate $\mathcal{T}^\ell$ by adding $e$ to  $\mathcal{T}^{\ell-1}$ if it results in a tree, 
				otherwise set $\mathcal{T}^{\ell}=\mathcal{T}^{\ell-1}$.
			\end{enumerate}}
		\end{enumerate}
	For every fixed $x,y \in V$ we define 
	\begin{equation}
	P_{\rm nice}^t(x,y)\coloneqq\text{prob. of 
		following a nice path $x \rightsquigarrow y$ of length $t$}.
	\end{equation}
		\end{definition}
	As mentioned above, nice paths \whp host typical trajectories of the SRW:
	\begin{proposition} \label{prop:nice-paths}
		Fix $\beta>1$ and $t=\beta\tent$. Then, uniformly in a starting position in $\veps$, defined in  Eq.\ \eqref{eq:Veps}, the SRW \whp follows nice paths. More precisely, 
        $$\max_{x \in \veps} \Big(1-\sum_{y \in V}P_{\rm nice}^t(x,y)\Big) = \op(1).$$
	\end{proposition}
	\begin{proof}
		Thanks to Lemma \ref{lemma:tree-like}, the requirements (ii) and (iii) of Definition \ref{def:nice-paths} are \whp satisfied. Then, letting 
        $\theta=({n \log^3 n})^{-1}$,
    	\begin{equation}
		\max_{x \in \veps} \Big(1-\sum_{y \in V}P_{\rm nice}^t(x,y)\Big) 
		\le \max_{x \in \veps} \quenchG_x\left(\mass(X_0,X_1,\dots,X_t)>\theta\right)+\op(1),
		\end{equation}
		and the \lhs (namely, the cost of requirement (i) ) vanishes by Theorem \ref{thm:LLN}.
	\end{proof}
	\subsection{Bounds on the total variation profile}
    With the latter tools at hand and with the help of Proposition \ref{prop:drop-positive-part} below, which is a rewriting of \cite[Prop. 3]{BP}, it is possible to bound effectively the distance to equilibrium.
	
	\begin{proposition} \label{prop:drop-positive-part} 
		Let $\alpha^{-1}\lesssim \tent$, $x\in V$, and $t = \s + \h + 1$ as in Definition \ref{def:nice-paths}.
		Then,
				\begin{equation}
			\p\left(
			P_{\rm nice}^t(x,y)
			\le (1+\delta)\nu_{c(x)}(y)+\frac{\delta}{n}, \, \forall x, \in \veps, \forall y \in V,
			\right)=1-o(1), \qquad \forall \delta>0,
		\end{equation}
		where $\semi^{\s+1}(\cdot,\cdot)$ is defined in Eq.~\eqref{eq:Q-def} and, for $i \le m$,
		\begin{equation} \label{eq:nu-i}
			\nu_i(y) \coloneqq \frac1n \sum_{w \in V}{\semi^{\s+1}(i,c(w)) }P^{\h}(w,y).
		\end{equation} 
	\end{proposition}
	\begin{proof}
    The proof mimics \cite[Prop. 3]{BP}.
		Let $x \in \veps$, as defined in Eq.\ \eqref{eq:Veps}, and let $\mathcal{F}$ denote the partial environment generated by the tree $\mathcal{T}_x(\s)$ and $\mathcal{B}^-_y(\h)$, the in-neighborhood of $y$ of depth $\h$. We split $P_{\rm nice}^t(x,y)$ in $m$ different addends. 
		For $i=1,\dots, m$, let
	$$P_{{\rm nice},i}^t(x,y) = 
	\sum_{z \in V^+_\mathcal{F}(i)} \sum_{v \in V^-_\mathcal{F}} 
	\mass(\pat_{x,z}) \frac{1}{D^{+}_z}\mass(\pat_{v,y})
	\one_{\{z \to v\}}\one_{\{\pat \textup{ is a nice path}\}},$$
	where $V^+_\mathcal{F}(i)$ is the set of vertices 
    at depth $\s$ in $\mathcal{T}_x(\s)\cap V_i$, and $V^-_{\mathcal{F}}$ is the set of vertices in $\partial \mathcal{B}^-_y(\h)$ such that there exists a unique path of length $\h$ to y.
		Then, it holds $P_{\rm nice}^t(x,y)=\sum_{i \le m}P_{{\rm nice},i}^t(x,y).$
		Setting $\p_{\mathcal{F}}(\cdot)\coloneqq\p(\cdot\,|\mathcal{F})$, we have that, for $z,v \in V$, 
			\begin{equation}
                \E\left[\frac{\one_{\{z \to v\}}}{D^{+}_z}\,\big|\mathcal{F}\right]=\E\left[\frac{\one_{\{z \to v\}}}{D^{+}_z}\,\big|\mathcal{F},\one_{\{z \to v\}}=1\right]\p\left((z,v)\in E\right)= \frac{p_{zv}}{np}(1+o(1)).
            \end{equation}
		Then,
		\begin{align*}
			\E[P_{{\rm nice},i}^t(x,y)\,|\, \mathcal{F}\,]
			&\le  \sum_{v \in V^-_\mathcal{F}}\sum_{z \in V^+_\mathcal{F}(i)} \mass(\pat_{x,z}) 
			\E\left[\frac{\one_{\{z \to v\}}}{D^{+}_z}\,\big|\mathcal{F}\right]\mass(\pat_{v,y})\\
			&\le  \sum_{v \in V^-_\mathcal{F}}\sum_{z \in V^+_\mathcal{F}(i)}\mass(\pat_{x,z}) \frac{p_{zv}}{np}\mass(\pat_{v,y})(1+o(1)).
		\end{align*}
		Since the random walk performs \whp a nice path  with support concentrated on the vertices of the tree $\mathcal{T}_x(\s)$ (Proposition \ref{prop:nice-paths}), and thanks to Proposition \ref{prop:quenched-law}, for $i \le m$, \whp it holds
    	\begin{equation}
		      \left|\sum_{z \in V^+_\mathcal{F}(i)} \mass(\pat_{x,z}) 
		      - \semi^ {\s} (c(x),i) \right|=o(1).
		\end{equation}
        Then, observing that $\frac{p_{zv}}{p}=\semi(c(z),c(v))$, it holds
			\begin{align*}
			{\E[P_{{\rm nice},i}^t(x,y)\,|\, \mathcal{F}\,]}
			&\le \semi^ \s (c(x),i)
			\sum_{v \in V^-_\mathcal{F}}  \frac{\semi(i,c(v))}{n}\mass(\pat_{v,y})(1+o(1)) \\
			&\le \semi^{\s} (c(x),i)  \sum_{w \in V} \frac{\semi (i,c(w))}{n} P^{\h}(w,y)(1+o(1)).
		\end{align*}
		Summing over $i \le m$, we conclude
		\begin{align} \label{eq:conditional-average}
			\E[P_{\rm nice}^t(x,y)\,|\, \mathcal{F}\,] \le \sum_{w \in V}\frac{\semi^{\s+1}(c(x),c(w)) }{n} P^{\h}(w,y)(1+o(1))= \nu_{c(x)}(y)(1+o(1)).
		\end{align}
		The proof then continues as in \cite[Prop. 3]{BP}, employing Eq.~\eqref{eq:conditional-average} 
        and a suitable Bernstein's concentration inequality,
        and averaging over the partial environment $\mathcal{F}$ to obtain the thesis.
	\end{proof}

	We now proceed with our analysis, pursuing an upper bound and a lower bound for the distance to stationarity at times $t$ greater than $\beta \tent$ for $\beta>1$, and smaller than $\beta\tent$ for $\beta<1$, respectively.
    
	\subsubsection{Upper bound} We start stating the following straightforward lemma.    
    \begin{lemma}
	\label{lemma:ciritcal-1/m}
	Let $\alpha^{-1}=C\tent$. Then for each $j=1,\dots, m$, it holds $\pi(V_j)=\frac1m(1+\op(1))$.
\end{lemma}
\begin{proof}
	By definition of stationarity, for each $j=1,\dots,m$, and for $t=(\log(n))^2$,
	\begin{equation}
		\left|\pi(V_j)-\frac1m\right|
		\le \sum_{x \in V}\pi(x)\left|\quenchG_x(X_t \in V_j)-\frac1m\right| \le \max_{x \in V} \left|\quenchG_x(X_t \in V_j)-\frac1m\right|,
	\end{equation}
	which is $\op(1)$ by Proposition \ref{prop:quenched-law}.
\end{proof}

    Choose now a time $t \ge \beta \tent$, for $\beta>1$. Applying Proposition \ref{prop:drop-positive-part}, and later Proposition \ref{prop:nice-paths}, for every $\delta>0$ it holds
	\begin{equation} 
		\begin{split}\label{eq:drop}
		\max_{x \in \veps}\|P^{t}(x,\cdot)-\nu_{c(x)}\|_{\textup{TV}}&\le
		\max_{x \in \veps}\sum_{y \in V}\left[(1+\delta)\nu_{c(x)}(y)+\frac{\delta}{n}-P_{\rm nice}^t(x,y)\right]\\
		&\le \max_{x\in\veps}\left(1-\sum_{y\in V}P_{\rm nice}^t(x,y)\right)+2\delta=2\delta+\op(1).
		\end{split}
	\end{equation}

	Then, thanks to Lemma \ref{lemma:tree-like} and Eq.~\eqref{eq:drop}, for $\ell=3 \log \log(n)$,
	\begin{equation} \label{eq:bound}
		\max_{x \in V}\|P^{t+\ell}(x,\cdot)-\nu_{c(x)}\|_{\textup{TV}} \le \quenchG_x(X_\ell \notin \veps) + \max_{x \in \veps}\|P^{t}(x,\cdot)-\nu_{c(x)}\|_{\textup{TV}} =\op(1).
	\end{equation}

Since, for every $w \in V$, it holds $\sum_{i =1}^m{\semi^ {\s+1}(i,c(w))}=1$, we can define
	\begin{equation}\label{eq:nu}
		\nu \coloneqq \frac 1{m} \sum_{i =1}^m\nu_i  =  \frac{1}{mn} \sum_{w \in V}P^\h(w,\cdot).
	\end{equation} 
	 Employing Lemma \ref{lemma:ciritcal-1/m} and later Eq.~\eqref{eq:bound}, it holds
		\begin{equation} \label{eq:nu-pi}
			\begin{split}
				\tv { \nu-\pi} 
				& = \tv { \frac1m\sum_{i=1}^m \nu_i-\sum_{z \in V} \pi(z)P^{(1+\eps)\tent}(z,\cdot)} \\
				& = \tv { \sum_{i=1}^m\sum_{z \in V_i} {\pi(z)} \left[\nu_i-P^{(1+\eps)\tent}(z,\cdot)\right]}+\op(1) \\
				& \le \max_{i \le m}\max_{z \in V_i} \tv {P^{(1+\eps)\tent}(z,\cdot)- \nu_{c(z)}}=\op(1).
			\end{split}
		\end{equation}
    This means that $\nu$ constitutes a good proxy for $\pi$. 
    Then, by the triangle inequality, and thanks to Eqs.~\eqref{eq:bound} and \eqref{eq:nu-pi}, we can conclude the following upper bound
\begin{equation}  \label{eq:bound2}
    \begin{split}
	\max_{x \in V}\tv{P^{t+\ell}(x,\cdot)-\pi} 
    &\le  \max_{x \in V}\tv{P^{t+\ell}(x,\cdot)-\nu_{c(x)}}+\max_{i \le m}\tv{\nu_i-\nu}+\tv{\nu-\pi}\\
    & \le \max_{i \le m}\tv{\nu_i-\nu}+\op(1).
    \end{split}
\end{equation}
    
	\subsubsection{Lower bound} Let $t \le \beta \tent$, for $\beta<1$.
    We are going to show that the law of $X_t$ at a time $t \le {\beta\tent}$, is \whp concentrated on a set with cardinality $o(n)$. 
	To this end, let $\theta=n^{-\beta(2-\beta)}$ and let
	$$
	S_x\coloneqq \left\{y \in V : \text{ there exists a path $\pat$ of length $t$ such that } \mass(\pat)\ge \theta\right\}.
	$$
	We have $|S_x|\le \theta^{-1}=n^{\beta(2-\beta)}=o(\frac{n}{\log(n)})$, since, for $\beta<1$, we have $\beta(2-\beta)<1$. Moreover, it holds 
	$
	-\frac{\log \theta}{\entropy t}=2-\beta>1,
	$ and we conclude the following lower bound:
	\begin{equation} \label{eq:lower-bound}
		\begin{split}
			\min_{x \in V} \tv { \quenchG_x(X_t \in \cdot)-\pi} 
			& \ge \min_{x \in V} \left[\quenchG_x(X_t \in S_x)-\pi(S_x)\right]\\
			& \ge \min_{x \in V}\quenchG_x\left(\mass(X_0,\dots,X_t) \ge n^{-\beta(2-\beta)}\right)-\max_{x \in V}\pi(S_x)\\
			&=1-\op(1),
		\end{split}
	\end{equation}
	where the last inequality holds combining Theorem \ref{thm:LLN} with the observation that $$\max_{x \in V}\pi(S_x)=\op(1)\,,$$ thanks to Theorem \ref{thm:pi-char} and Proposition \ref{prop:concentration-degrees}.

\section{Proof of the main result }\label{sec:proof}
\subsection{Supercritical regime}
In this regime, the proof of Theorem \ref{thm:main} is split into two parts, depending on the chosen timescale. We first analyze the case $t \asymp \tent$, proving Eq.~\eqref{eq:sup-cutoff}, and then move to the case $t \asymp \alpha^{-1}$, proving Eq.~\eqref{eq:sup-meta}.
    \subsubsection{Relaxation to a local equilibrium}
 Let us first assume
$\tent\ll\alpha^{-1} \ll \sqrt{n}\log(n)^{-2}$. We will sometimes commit a slight abuse of notation by lifting $\pi_i$ to a probability measure on the entire vertex set $V$.
Let $t=\beta\tent$ for some $\beta<1$. 
By Theorem \ref{thm:cutoff},
\begin{equation} \label{eq:bound-pi_i}
	\min_{i \le m}\min_{x\in V_i}\|\quench^{G_i}_x(X^{i}_t\in\cdot)-\pi_{i}\|_{\rm TV}=1-\op(1)\,,
\end{equation}
In particular, for every $\delta>0$ and $x \in V_i$, it must hold 
\begin{equation}
    \frac 12 \sum_{y \in V}\left|\quench^{G_i}_x(X^i_t=y)-\frac1m\pi_i(y)\right|\ge 1-\frac{m-1}{2m}-\delta + \op(1),
\end{equation} otherwise
 there would exists a $\delta>0$ such that
\begin{equation}
\begin{split}
    \left\|\quench^{G_i}_x(X^i_t\in\cdot)-\pi_i\right\|_{\rm TV}
    & \le \frac 12 \sum_{y \in V}\left|\quench^{G_i}_x(X^i_t=y)-\frac1m\pi_i(y)\right| + \frac12\cdot\frac{m-1}{m}\sum_{y \in V}\pi_i(y)\\
       &<1-\frac{m-1}{2m}-\delta+\frac{m-1}{2m}+\op(1)=1-\delta+\op(1),
\end{split}
\end{equation}
which is in contradiction with \eqref{eq:bound-pi_i}.
As a consequence, using the characterization of $\pi$ given in Theorem \ref{thm:pi-approximation}, and using that, by Proposition \ref{prop:jump-time}, 
\begin{equation}
\begin{split}
    \left\|\quenchG_x(X_t\in\cdot)-\quench^{G_i}_x(X_t^i\in \cdot)\right\|_{\rm TV} &
    \le \max_{i \le m}\max_{x\in V_i}\check\quench_x(X_t{\neq} X_t^{i})\\
    &\le \max_{i \le m}\max_{x\in V_i}\quench^G_x(\tauj\le t)=\op(1),
    \end{split}
\end{equation}
we have that, for every $\delta > 0$ and $x \in V_i$, 
\begin{equation} \label{eq:lower-1}
\begin{split}
    \left\|\quenchG_x(X_t\in\cdot)-\pi\right\|_{\rm TV}
    &=\left\|\quench^{G_i}_x(X^i_t\in\cdot)-\frac1m \sum_{j =1}^m \pi_j\right\|_{\rm TV}+\op(1)\\
    &=\frac 12 \sum_{y \in V}\left|\quench^{G_i}_x(X^i_t=y)-\frac1m\pi_i(y)\right| + \sum_{j \neq i}\frac{1}{2m}\sum_{y \in V}\pi_j(y)+\op(1)\\
    & \ge 1-\frac{m-1}{2m}-\delta+\frac{m-1}{2m}+\op(1)=1-\delta+\op(1).
\end{split}
\end{equation}
Let now $\beta>1$. By definition of total variation distance
\begin{equation} \label{eq:lower-bound-1-alpha}
\begin{split}
	\min_{i \le m}\min_{x\in V_i}\|\quenchG_x(X_t\in\cdot)-\pi\|_{\rm TV} &\ge\min_{i \le m}\min_{x\in V_i}\left| \quenchG_x(X_t\in V_{i})-\pi(V_{i})\right|
	\\
	&= \frac{m-1}{m}\left(1-\frac{m}{m-1}\alpha\right)^{t-1}+\op(1)\,,
\end{split}
\end{equation}
where we have used Proposition \ref{prop:quenched-law} and the characterization of $\pi$ given in Theorem \ref{thm:pi-approximation}. 
By our choice $t=\beta\tent \ll \alpha^{-1}$, one gets
\begin{equation} \label{eq:lower-bound-m-1}
\min_{i \le m}\min_{x\in V_i}\|\quenchG_x(X_t\in\cdot)-\pi\|_{\rm TV} \ge \frac{m-1}{m}-\op(1)\,.
\end{equation}
For what concerns the upper bound, for $\beta>1$, let us fix some $\gamma\ge 0$, possibly depending on $n$, and $\varepsilon>0$ such that $(1+\varepsilon)\tent+\gamma \le \beta\tent$. By splitting over the vertex on which the SRW sits at time $(1+\varepsilon)\tent$, and over the community of such vertex one gets, for $x \in V$,
\begin{equation}\label{eq:1}
	P^{(1+\varepsilon)\tent+\gamma}(x,\cdot)=\sum_{y \in V}P^{\gamma}(x,y)P^{(1+\varepsilon)\tent}(y,\cdot)=\sum_{i=1}^m\sum_{y \in V_i}P^{\gamma}(x,y)P^{(1+\varepsilon)\tent}(y,\cdot)\,.
\end{equation}
Using Theorem \ref{thm:cutoff-community}, we can write
\begin{align}
	\max_{x \in V} 	\tv{ P^{(1+\varepsilon)\tent+\gamma}(x,\cdot)  - \sum_{i=1}^mP^{\gamma}(x,V_i) \pi_i} =\op(1).
\end{align}
Let us now focus on the case $\tent\ll\alpha^{-1} \ll \sqrt{n}\log(n)^{-2}$. By Proposition \ref{prop:quenched-law} we get
\begin{align}\label{eq:439}
	\max_{x \in V} 	\tv {P^{(1+\varepsilon)\tent+\gamma}(x,\cdot) 
		-\sum_{i=1}^m \semi^\gamma(c(x),i) \pi_i} = \op(1).
\end{align}
Thanks to Theorem \ref{thm:pi-approximation} and recalling the definition in \eqref{eq:Q-def}, we then obtain
\begin{equation} \label{eq:gamma}
	\begin{split}
	\max_{x \in V} 	\tv{ {P^{(1+\varepsilon)\tent+\gamma}(x,\cdot) - \pi} }
	&= \frac12\sum_{i=1}^m    \left|\semi^\gamma(c(x),i)-\frac 1m \right|+\op(1)\\
	& =\frac{(m-1)(1-\frac{m}{m-1}\alpha)^{\gamma} }{m}+\op(1).
	\end{split}
\end{equation}
Taking $\gamma\ll\alpha^{-1}$ and by monotonicity, we conclude that 
\begin{equation}  \label{eq:upper-bound-m-1}
	\max_{x \in V} 	\tv{ {P^{\beta\tent}}(x,\cdot) - \pi } \le\frac{m-1}{m}+\op(1).
\end{equation}

Then, putting together Eqs.~\eqref{eq:lower-1}, \eqref{eq:lower-bound-m-1}, and \eqref{eq:upper-bound-m-1}, Eq.~\eqref{eq:sup-cutoff} follows.
Similarly, in the case $\sqrt{n}\log(n)^{-2}\ll\alpha^{-1} \ll \lambda n \log(n) $, instead of Proposition \ref{prop:quenched-law} and Theorem \ref{thm:pi-approximation}, we use Proposition \ref{prop:quenched-law-new} and Theorem \ref{thm:pi-approximation-new}, leading to the same result.

\subsubsection{Convergence to the global equilibrium}
Let us first assume
$\tent\ll\alpha^{-1} \ll \sqrt{n}\log(n)^{-2}$. Using Eq.~\eqref{eq:lower-bound-1-alpha}, with $t=\beta\alpha^{-1}$, for $\beta>0$,
\begin{equation}
	\min_{i \le m}\min_{x\in V_i}\|\quenchG_x(X_t\in\cdot)-\pi\|_{\rm TV} \ge \frac{m-1}{m}{\rm e}^{-\frac{\beta m}{m-1}}+\op(1)\,.
\end{equation}
For the upper bound, consider now, in Eq.~\eqref{eq:gamma}, $\gamma=t-(1+\eps)\tent \sim\beta\alpha^{-1}$. Then, again by monotonicity, we have
\begin{equation}
	\max_{x\in V}\|\quenchG_x(X_t\in\cdot)-\pi\|_{\rm TV} \le \frac{m-1}{m}{\rm e}^{-\frac{\beta m}{m-1}}+\op(1)\,,
\end{equation}
concluding the proof. Similarly, in the case $\sqrt{n}\log(n)^{-2}\ll\alpha^{-1} \ll\lambda n \log(n) $, instead of Proposition \ref{prop:quenched-law} and Theorem \ref{thm:pi-approximation}, we use Proposition \ref{prop:quenched-law-new} and Theorem \ref{thm:pi-approximation-new} , leading to the same result. This proves Eq.~\eqref{eq:sup-meta}.

\subsection{Subcritical regime}
{The upper bound} in Eq.~\eqref{eq:sub-cutoff} (which is non-trivial only for $\beta>1$) can be obtained observing that in the subcritical case $\alpha^{-1} \ll \tent$, it holds
\begin{equation} 
	{\semi^{\s}(i,\cdot)}=\frac{1}{m}(1+o(1)), \qquad \forall i \le m,
\end{equation}
Then, recalling the definitions in Eqs.~\eqref{eq:nu-i} and \eqref{eq:nu}, it follows $\max_{i \le m}\tv{\nu_{i}-\nu}=\op(1)$. This, plugged in Eq.~\eqref{eq:bound2},
 concludes the upper bound.

The lower bound in Eq.~\eqref{eq:sub-cutoff} (which is non-trivial only for $\beta<1$) is precisely Eq.~\eqref{eq:lower-bound}.
    \subsection{Critical regime}
	We fix $C>0$ and $t=\beta\tent$, for some $\beta>0$, $\beta\neq1$. We want to prove Eq.~\eqref{eq:crit-cutoffmeta}. For $\beta<1$, the bound is, again, precisely given by Eq.~\eqref{eq:lower-bound}.
    Then, we consider $\beta>1$. By a trivial bound and Proposition \ref{prop:quenched-law},
	\begin{equation}
		\begin{split}
			\min_{x \in V}\tv { P^{\beta\tent}(x,\cdot)-\pi }
			 &\ge  \min_{x \in V}\left|P^{\beta\tent}(x,V_{c(x)})-\pi(V_{c(x)})\right|\\
			& =\tfrac{m-1}m \left(1-\tfrac{m}{m-1}\alpha\right)^{\beta \tent}+\op(1)\\
            &=\tfrac{m-1}{m}{\rm e}^{-\frac{\beta}{C}\frac{m}{m-1}}(1+o(1))+\op(1).
		\end{split}
	\end{equation}
	In the what follows we will prove the following result.
    \begin{lemma} \label{lemma:desired}
        For every $0<\eps<\beta-1$, 	
            \begin{equation} 
		              \max_{x \in V}\tv { P^{\beta\tent}(x,\cdot)-\pi } 
                      \le \tfrac{m-1}{m}{\rm e}^{-\frac{\beta}{C}\frac{m}{m-1}}(1+o(1))+   \frac{6m\eps}{C}+\op(1).
                \end{equation}
    \end{lemma}
    Since $\eps>0$ can be taken arbitrarily small, this lemma closes the discussion for $\beta>1$, finishing the proof of Theorem \ref{thm:main}.
	    To prove it, it is useful to consider, for every $i \le m$, the measures
$$\pi_{V_i}(x)=\frac{\pi(x)\one_{\{x\in V_i\}}}{\pi(V_i)}\,, \qquad x \in V.
$$
\begin{lemma} \label{lemma:distance-eps} 
	Let $\alpha^{-1}=C \tent $, for $C> 0 $. Then
	\begin{equation}
	\p\left(\max_{i \le m}\tv { \frac1n  \sum_{y \in V_i}P^{\eps \tent}(y,\cdot )-{\pi_{V_i}} }\le \frac{4m\eps }{C}\right)=1-o(1).
	\end{equation}
\end{lemma}
\begin{proof}[Proof of Lemma \ref{lemma:distance-eps}]
		Using the inequality $|a+b|\ge |a|-|b|$, for $a,b \in \R$, we have
	\begin{equation}  \label{eq:distance-eps}
		\begin{split}
			\Bigg\| \frac{1}{mn}&\sum_{y \in V}P^{\h}(y,\cdot)-\frac{1}{m}\sum_{i=1}^m\pi_{V_i} \Bigg\|_{\rm TV}\\
			&= \frac12\sum_{j=1}^m\sum_{z \in V_j}\frac1m \left|\frac{1}{n}\sum_{y \in V_j}P^{\h}(y,z)-\sum_{i=1}^m{\pi_{V_i}(z)}+\frac{1}{n}\sum_{y \in V_j^\complement}P^{\h}(y,z)\right|\\
			& \ge  \frac12\sum_{j=1}^m\sum_{z \in V_j}\frac1m \left[\left|\frac{1}{n}\sum_{y \in V_j}P^{\h}(y,z)-\pi_{V_j}(z)\right| -\frac{1}{n}\sum_{y \in V_j^\complement}P^{\h}(y,z)\right]\\
			&= \frac12\sum_{j=1}^m\frac1m \left[\sum_{z \in V_j}\left|\frac{1}{n}\sum_{y \in V_j}P^{\h}(y,z)-{\pi_{V_j}(z)}\right| -\frac{1}{n}\sum_{y \in V_j^\complement}P^{\h}(y,V_j)\right].
		\end{split}
	\end{equation}
	On the other hand, for every $j \le m$ it holds
	\begin{equation} \label{eq:distance-j-eps}
		\begin{split}
			\tv { \frac1n \sum_{y \in V_j} P^{\h}(y,\cdot )-{\pi}_{V_j} }\!\!\!\!
			&=\frac12\sum_{z \in V} \left| \frac1n \sum_{y \in V_j}P^{\h}(y,z)-{\pi}_{V_j}(z)  \right|\\
		&=\frac12 \left[\sum_{z \in V_j} \left| \frac1n\sum_{y \in V_j}P^{\h}(y,z)-{\pi}_{V_j}(z)  \right|+ \frac1n \sum_{y \in V_j} P^{\h}(y, V_j^\complement )\right].
		\end{split}
	\end{equation}
	Putting Eqs.~\eqref{eq:distance-eps} and \eqref{eq:distance-j-eps} together, it follows
		\begin{equation} \label{eq:bound-eps}
		\begin{split}
			\sum_{j=1}^m\frac1m \tv {  \frac1n\sum_{y \in V_i} P^{\h}(y,\cdot )-{\pi}_{V_i} } \le \tv { \frac{1}{mn}\sum_{y \in V}P^{\h}(y,\cdot)-\frac{1}{m}\sum_{i=1}^m\pi_{V_i} }  + \Phi(\eps),
		\end{split}
	\end{equation}
	where
	\begin{equation}
		\begin{split}
			\Phi(\eps) &\coloneqq	\frac1{2m} \sum_{j=1}^m \frac{1}{n}\Big(\sum_{y \in V_j^\complement}P^{\h}(y,V_j)+\sum_{y \in V_j}P^{\h}(y,V_j^\complement)\Big)
			 = 	\frac1{m} \sum_{j=1}^m \frac{1}{n}\sum_{y \in V_j}P^{\h}(y,V_j^\complement).
		\end{split}
	\end{equation}
	
	Invoking Proposition \ref{prop:quenched-law}, we get $$\left|\Phi(\eps)-\frac1m\sum_{j=1}^m\semi^\h(j, [m]\setminus \{j\})\right|=\op(1).$$ Using that, for $\tent=\alpha^{-1}/C$, and for each $j \le m$,
	\begin{equation} \label{eq:Q}
		\begin{split}
			 \semi^\h\left(j,[m]\setminus\{j\}\right)
             &=\frac{m-1}{m}\left(1-\left(1-\tfrac m{m-1}\alpha\right)^\h\right)\\
             &= \frac{m-1}{m}\left(1- {\rm e}^{-\frac{m}{m-1}\frac{2\eps}{C}}(1+o(1))\right) \le \frac{2\eps}{C}(1+o(1)),
		\end{split}
	\end{equation}
    As a consequence, \whp $\Phi(\eps) \le \frac{3\eps}{C}$.
	Recalling that $\pi=\frac{1}{m}\sum_{j=1}^m\pi_{V_j}(1+\op(1))$, thanks to Lemma \ref{lemma:ciritcal-1/m}, from Eq.~\eqref{eq:bound-eps} we conclude that \whp it holds 
	\begin{equation}
		\begin{split} 
			\frac1m\sum_{j=1}^m \tv {  \sum_{y \in V_j}\frac1n P^{\h}(y,\cdot )-{\pi}_{V_j} } \le \tv { \frac{1}{mn}\sum_{y \in V}P^{\h}(y,\cdot)-\pi} +\frac{3\eps}{C},
		\end{split}
	\end{equation}
	which is at most $\frac{3\eps}{C}+\op(1)$, thanks to Eq.~\eqref{eq:nu-pi}. Then, w.h.p., each term of the sum is bounded by $\frac{4m\eps}{C}$.
\end{proof}

\begin{proof}[Proof of Lemma \ref{lemma:desired}] 
We consider a time $t=\beta\tent$,
where $\eps>0$ is chosen such that $\beta>1+\eps$. For $x \in V$,
it holds
\begin{equation}
	\begin{split}
		\tv { P^{t}(x,\cdot)-\pi} 
		& \le \tv { P^{t}(x,\cdot)-\sum_{j=1}^m \semi^{t}(c(x),j) \pi_{V_j}}   \!\!\!\! +
		\tv { \sum_{j=1}^m \semi^{t}(c(x),j) {\pi_{V_j}}-\pi}.
	\end{split}
\end{equation}
Being $\pi= \sum_{j=1}^m\pi(V_j) {\pi_{V_j}}$, taking the maximum over $x\in V$, the second summand provides
\begin{equation} \label{eq:spirit}
	\begin{split}
		\max_{x \in V}&\tv { \sum_{j=1}^m \semi^{t}(c(x),j){\pi_{V_j}} -\sum_{j=1}^m\pi(V_j) {\pi_{V_j}} }
		= \max_{x \in V}\frac12 \sum_{j=1}^m \left|\semi ^{t}(c(x),j) -\pi(V_j)\right|\\
		& \le  \max_{x \in V}\frac12 \sum_{j=1}^m \left|\semi^{t}(c(x),j) -\frac 1m\right|+\op(1)
		=\frac{m-1}m \left(1-\frac{m}{m-1}\alpha\right)^{t}+\op(1),
	\end{split}
\end{equation}
where we have employed Lemma \ref{lemma:ciritcal-1/m} and the triangle inequality in the second line. This will provide the leading term. 
We now consider the first summand.
By Proposition \ref{prop:quenched-law}, it holds
\begin{equation}
	\max_{x \in V}\max_{j \le m}\left|P^{(\beta-1-\eps)\tent}(x,V_j) - \semi ^{(\beta-1-\eps)\tent} (c(x),j)\right|=\op(1).
\end{equation}
Then
\begin{equation} \label{eq:bounds}
	\begin{split}
		&\max_{x \in V}\left\| P^{\beta \tent}(x,\cdot)- \sum_{j=1}^m  \semi^{\beta \tent}(c(x),j) \pi_{V_j} \right\|_{TV} \\
		& \le \max_{x \in V}\tv{ \sum_{i=1}^m \sum_{z \in V_i} P^{(\beta-1-\eps)\tent}(x,z) \left[P^{(1+\eps)\tent}(z,\cdot)- \sum_{j=1}^m \semi^{(1+\eps)\tent}(i,j)\pi_{V_j} \right]}+\op(1)\\
		& \le \max_{i \le m} \max_{z \in V_i} \tv{ P^{(1+\eps)\tent}(z,\cdot)- \sum_{j=1}^m \semi^{(1+\eps)\tent}(i,j)\pi_{V_j}}+\op(1) \\
		& \le \max_{i \le m}  \tv{ \sum_{k=1}^m\frac{\semi^{\s+1}(i,k)}{n}\sum_{y \in V_k}P^{\h}(y,\cdot)- \sum_{j=1}^m \semi^{(1+\eps)\tent}(i,j)\pi_{V_j}}+\op(1) 
        \end{split}
\end{equation}
where the third inequality follows from Eq.~\eqref{eq:bound} (recall that, according to Definition \ref{def:nice-paths}, $(1+\eps)\tent=\s+\h+1$) Using the semigroup property of $\semi(\cdot,\cdot)$,
we can bound the last expression in Eq.~\eqref{eq:bounds} by
\begin{equation} \label{eq:}
	\begin{split}
		& \max_{i \le m}  \sum_{k=1}^m \semi^{\s+1}(i,k) \tv{ \frac1n\sum_{y \in V_k}P^{\h}(y,\cdot)- \sum_{j=1}^m \semi^{\h}(k,j)\pi_{V_j}}+\op(1) \\
		& \le \max_{k \le m}\left(\tv{ \frac1n\sum_{y\in V_k} P^{\h}(y,\cdot)-\pi_{V_k}}+\tv{\pi_{V_k}-\sum_{j=1}^m \semi^{\h}(k,j)\pi_{V_j}} \right)+\op(1),
	\end{split}
\end{equation}
Now \whp we can bound the first summand in the last display with $\frac{4m\eps}{C}$ by Lemma \ref{lemma:distance-eps}. The second summand has the same magnitude: reasoning in the spirit of Eqs.~\eqref{eq:spirit} and \eqref{eq:Q}, it holds
\begin{equation}
	\begin{split}
		\max_{k \le m}\tv { {\pi_{V_k}} - \sum_{j=1}^m \semi^{\h}(k,j){\pi_{V_j}} }
		&=\max_{k\le m} \frac12 \left|1-\semi ^{\h}(k,k)\right|+\frac12 \sum_{j\neq k} \semi^{\h}(k,j) \\
		& = \max_{k\le m} \frac{m-1}{m}\left(1-\left(1-\tfrac m{m-1}\alpha\right)^\h\right)
		\le \frac{2\eps}{C}(1+o(1)). \qedhere
	\end{split}
\end{equation}
\end{proof}

\end{document}